\documentclass[11pt]{amsproc}
\usepackage{amsmath,amsfonts,amsthm,amssymb,amscd,hyperref,url}
\usepackage[english]{babel}
\usepackage[T1]{fontenc}
\usepackage[utf8x]{inputenc}  
\def\classification#1{\def\@class{#1}}
\classification{\null}

\DeclareFontFamily{OT1}{rsfs}{}
\DeclareFontShape{OT1}{rsfs}{n}{it}{<-> rsfs10}{}
\DeclareMathAlphabet{\mathscr}{OT1}{rsfs}{n}{it}

\DeclareMathOperator{\mo}{\,mod}

\DeclareMathOperator{\GL}{GL}
\DeclareMathOperator{\SO}{SO}
\DeclareMathOperator{\diam}{diam}

\DeclareMathOperator{\SL}{SL}

\DeclareMathOperator{\PSL}{PSL}

\DeclareMathOperator{\PSp}{PSp}

\DeclareMathOperator{\Sym}{Sym}
\DeclareMathOperator{\Ad}{Ad}

\DeclareMathOperator{\Stab}{Stab}
\DeclareMathOperator{\Alt}{Alt}

\DeclareMathOperator{\charac}{char}

\DeclareMathOperator{\rank}{rank}

\DeclareMathOperator{\tr}{tr}

\DeclareMathOperator{\Cl}{Cl}

\DeclareMathOperator{\BS}{BS}

\DeclareMathOperator{\un}{un}

\newtheorem{theorem}{Theorem}[section]
\newtheorem{lem}[theorem]{Lemma}
\newtheorem{prop}[theorem]{Proposition}
\newtheorem{prob}[theorem]{Exercise}
\newtheorem{coun}{Counterexample}
\newtheorem{conj}{Conjecture}
\newtheorem{defn}{Definition}[section]
\newtheorem{corollary}[theorem]{Corollary}
\numberwithin{equation}{section}
\title{Growth and expansion in algebraic groups over finite fields}
\author{Harald Andr\'es Helfgott}
\address{Harald A. Helfgott, 
  Mathematisches Institut,
Georg-August Universit\"{a}t G\"{o}ttingen, Bunsenstra{\ss}e 3-5, D-37073 G\"{o}ttingen,
Germany; IMJ-PRG, UMR 7586,
  58 avenue de France, B\^{a}timent S. Germain, case 7012,
  75013 Paris CEDEX 13, France}
\email{harald.helfgott@gmail.com}
\begin{document}
\maketitle
\tableofcontents

\section{Introduction}
%Provisos: title of school, title of notes (groups of Lie type)

This text is meant to serve as a brief introduction to the study
of growth in groups of Lie type, with $\SL_2(\mathbb{F}_q)$ and some of
its subgroups as the key examples.
 They are an edited version of the notes I distributed at the
 Arizona Winter School in 2016. Those notes were, in turn,  
 based in part on the survey \cite{MR3348442} and in part
 on the notes for courses I gave on the subject in Cusco \cite{AGRA2}
 and G\"ottingen. 

Given the format of the Arizona Winter School, the emphasis here is on reaching
the frontiers of current research as soon as possible, and not so much
on giving a comprehensive overview of the field. For that the reader is
referred to \cite{MR3348442} and its bibliography, or
to \cite{MR3144176} and \cite{MR3309986}. At the same time -- again motivated
by the school's demands -- we will take a brief look at several
applications at the end.

It will be necessary to be minimally conversant with some of the basic
classical vocabulary of algebraic geometry (as in the first
chapter of Mumford's Red Book \cite{MR1748380}), and with some 
notions on algebraic groups (such as $\SL_2$) and
Lie algebras (such as $\mathfrak{sl}_2$). A very brief compendium of what will
be needed can be found in \S \ref{subs:prelim}. It is often helpful
(and only rarely misleading) to be willing to believe that matters
work out in much the same way over finite field as they do over the reals.

%In essenbeing prepared to accept that 
%you can pass pass from $\SL_2$ to $\mathfrak{sl}_2$, say, and back,
%not just when you are working over $\mathbb{R}$ and $\mathbb{C}$, but also
%when you work over finite fields, where what is meant by derivation is
%non-obvious. We will go back an

The purpose of these notes is expository, not historical, though I have
tried to give key references. The origins of several ideas are traced in
greater detail in \cite{MR3348442}. In \S \ref{subs:overv}, we will give
a summary of the results we later prove and also of results and open questions
of the same kind. We will go over some important 
related questions and applications later, in \S \ref{sec:further}.

{\em Acknowledgements.} I was supported by ERC Consolidator grant
648329 (codename GRANT) and by funds from my Humboldt professorship.
Many thanks are due to a helpful and spirited anonymous referee.
Thanks are due as well to Lifan Guan, for providing a useful
reference and catching several typos, and to the audiences both
at the Arizona Winter School and at the Hausdorff Institute (HIM), for
real-time feedback.

\subsection{Basic questions and concepts: diameter, growth, expansion}\label{subs:whatisg}

Let $A$ be a finite subset of a group $G$. Consider the sets
\[\begin{aligned}
A&,\\
A\cdot A &= \{x\cdot y : x,y\in A\},\\
A\cdot A\cdot A &= \{x\cdot y\cdot z: x,y,z\in A\},\\
&\dotsc\\
A^k &= \{x_1 x_2 \dotsc x_k : x_i\in A\}.
\end{aligned}\]
Write $|S|$ for the {\em size} of a finite set $S$, meaning simply the number
of elements of $S$. A question arises naturally: how does $|A^k|$ grow as $k$
grows?

This kind of question has been studied from the perspective of additive 
combinatorics (for $G$ abelian) and geometric group theory ($G$ infinite,
$k\to \infty$). There are also some crucial related concepts coming
from other fields: {\em diameters} and {\em expanders}, to start with.

{\em Diameters.} Let $A$ be a set of generators of $G$.
When $G$ is infinite, a central question is how $|A^k|$
behaves as a function of $k$ as $k\to \infty$. When $G$ is finite, that
question does not make much sense, as $|A^k|$ obviously stays constant
as soon as $A^k = G$. Instead, let us ask ourselves what is the least
value of $k$ such that $A^k = G$. This value of $k$ is called the
{\em diameter}. It is finite because, for $A$ generating $G$,
$A^j\ne G$ implies $|A^{j+1}|>|A^j|$.
(Why is this last statement true?)

The term {\em diameter} comes from geometry. What we have is not just an analogy
-- we can actually put our basic terms in a geometrical framework, as
geometric group theory does. A {\em Cayley graph}  $\Gamma(G,A)$ is the graph
having $V=G$ as its set of vertices and $E=\{(g,ag): g\in G, a\in A\}$ as its set of edges.
Define the length of a path in the graph as the number of edges
in it, and the distance $d(v,w)$ between two vertices $v$, $w$
in the graph as the length
of the shortest path between them. The {\em diameter} of a graph 
is the maximum of the distance
$d(v,w)$ over all vertices $v$, $w$. It is easy to see
that the diameter of $G$ with respect to $A$, as we defined it above, equals
the diameter of the graph $\Gamma(G,A)$.

%It is clear, then, that showing that $A^k$ grows rapidly is a natural route
%towards bounds on the diameter.

{\em Product theorems.}
A central question of additive combinatorics is as follows: for finite
subsets 
$A$ of an abelian group $(G,+)$, when exactly is
it that $A+A$ is much larger than $A$?
In non-abelian groups $(G,\cdot)$, the right form of the question turns out to
be: given a set of generators $A$ of $G$, when is $A^3$
much larger than $A$? (We will see later why it is better to ask
about $A^3=A\cdot A\cdot A$ rather than $A^2 = A\cdot A$ here.)

It is clear that, if we show that, for any generating set $A$ of $G$,
\begin{equation}\label{eq:dichot}
  \text{either}\;\;\;\;\; \text{$|A^3|$ is much larger than $A$}\;\;\;\;\;\;
  \text{or}\;\;\;\; A^3 = G,
\end{equation}
then $A^k$ grows rapidly until roughly the point where $A^k = G$: simply
apply (\ref{eq:dichot}) to $A$, $A^3$, $A^9$, etc., in place of $A$.
In particular, (\ref{eq:dichot}) yields an upper bound on the
diameter of $G$ with respect to $A$. We call a result of the form
(\ref{eq:dichot}) a {\em product theorem}.

{\em Expansion.}
We say that a graph is an {\em vertex expander} with parameter $\delta>0$
(or {\em $\delta$-vertex expander})
if, for every subset $S$ of the set of vertices $V$ satisfying (say)
$|S|\leq |V|/2$, the number of vertices $v\in V$ not in $S$ such that
at least one edge connects $v$ to some element of $S$ is at least $\delta |S|$.
(We may think of $S$ as being a set of infected individuals; then we are
saying that the number of the newly infected will always
be at least $\delta |S|$,
unless the disease has reached a near-saturation point.)

Two closely connected notions are that of {\em edge expansion} and
{\em spectral expansion}.
First, some basic terms. A graph is {\em regular} if, for any vertex $v$,
the number of vertices $w$ such that $(v,w)$ is an edge equals a
constant $d$, and the number of vertices $w$ such that $(w,v)$ is an
edge also equals a constant (which must also be $d$, by a simple counting
argument). We call $d$ the {\em degree} or {\em valency} of the graph.
A Cayley graph $\Gamma(G,A)$ is always regular of degree $d=|A|$.

A regular graph $\Gamma=(V,E)$
of degree $d$ is a $\delta$-{\em edge expander} if, for every
$S\subset V$ satisfying $|S|\leq |V|/2$, the number of edges having one vertex
in $S$ and one outside $S$ is at least $\delta d |S|$. It is clear that,
if $\Gamma$ is a $\delta$-vertex expander, then it is a $(\delta/d)$-edge expander, and, if it is a $\delta$-edge expander, then it is a $\delta$-vertex
expander.

We say that a graph $\Gamma$ is {\em symmetric} to mean that
$(v,w)$ is an edge if and only if $(w,v)$ is an edge.
If $\Gamma$ is a Cayley graph $\Gamma(G,A)$, then $\Gamma$ is symmetric
provided that  $A^{-1} = \{g^{-1} : g\in A\}$ equals $A$.
We will generally assume that $A^{-1}=A$ without much loss of generality.
(Replace $A$ by $A\cup A^{-1}$ otherwise.)

Given a regular graph $\Gamma$ with a set of vertices $V$, the
{\em adjacency operator} $\mathscr{A}$ is the linear operator
taking any given function $f:V\to \mathbb{C}$ to the function
$\mathscr{A}f:V\to \mathbb{C}$ defined by
\begin{equation}\label{eq:adjac}
  \mathscr{A} f(v) = \frac{1}{d} \sum_{w: \text{$(v,w)$ is an edge}} f(w).
\end{equation}
Assume that the graph $\Gamma$ is symmetric.
Then $\mathscr{A}$ is a symmetric operator,
and thus has full real spectrum. Its largest eigenvalue is $1$;
it corresponds to constant eigenfunctions.
If every eigenvalue $\lambda$ of $\mathscr{A}$ corresponding to
non-constant eigenfunctions satisfies
$\lambda\leq 1- \delta$ for some $\delta>0$, we say that $\Gamma$ is
a $\delta$-spectral expander, or a {\em $\delta$-expander} for short.

If a regular, symmetric graph 
is a $\delta$-spectral expander, then it is
a $(\delta/2)$-edge expander, and, if it is a $\delta$-edge expander, then it is a $(\delta^2/2)$-spectral expander. This fact is non-trivial; it
is called the {\em Cheeger-Alon-Milman} inequality \cite{MR782626}, by analogy
with the {\em Cheeger inequality} on manifolds
\cite{MR0402831}.

%effects within them ({\em vertex expansion}). Conversely, thanks to 
%%\cite{MR2415383} and other works in the same direction, we know
%that growth in the sense above -- namely, rapid growth of
%$|A^k|$ -- can be used to prove expansion in several contexts.
%We will not study expansion in detail here, but it is important to mention it,
%since it is a subject of great interest in its own right, and many applications
%of growth go through it.

The notion of spectral expansion is natural, not just because of the analogy
with surfaces and their Laplacians, but, among other reasons,
because of random walks: a drunken
mathematician left to wander in a spectral expander $\Gamma$ will be anywhere
with about the same probability after only a short while.
To put matters more formally --
as we shall see in \S \ref{sec:exprawsie}, spectral expansion
implies small {\em mixing time}.

Since the diameter of a graph is bounded by its ($\ell_\infty$-)mixing time,
it follows immediately that spectral expansion implies small diameter.
We can also prove this implication
going through edge and vertex expansion: if a graph is a $\delta$-vertex
expander, it is very
easy to see that its diameter is $\ll (\log |G|)/\delta$; apply, then,
the Cheeger-Alon-Milman inequality.
%for the lazy random walk

\subsection{A brief overview of results on growth and diameter}\label{subs:overv}

Let us first review some basic terms from group theory.
A group $G$ is {\em simple} if it has no normal subgroups other than itself
and the identity.
A {\em subnormal series} of a group $G$ is a sequence of subgroups
\begin{equation}\label{eq:turu}\{e\} = H_0 \triangleleft H_1 \triangleleft H_2 \
\triangleleft \dotsb 
\triangleleft H_k = G,\end{equation}
i.e., $H_i$ is normal in $H_{i+1}$ for every $0\leq i<k$.
A {\em decomposition series} is a subnormal series in which every quotient
$H_{i+1}/H_i$ is simple. It is clear that every finite group has a decomposition
series. 

In some limited sense, questions on growth behave well under taking quotients,
and thus reduce to the case of simple groups, at least if our
decomposition series of bounded length. (To be precise:
for how product theorems behave under taking quotients, see exercises
\ref{prob:bara} and \ref{prob:bere}). For the behavior of diameters under
quotients, look up {\em Schreier generators}.) It thus makes sense to focus
on simple groups.

\subsubsection{Simple groups: what to expect?}

Some special cases of the following conjecture are arguably older ``folklore''.
\begin{conj}\label{babaiconj}
  $($Babai, \cite[Conj. 1.7]{BS92}$)$
  Let $G$ be finite, simple and non-abelian. Let $A$ be any set of generators
  of $G$. Then
%\begin{equation}\label{eq:bs}
\[
\diam(\Gamma(G,A)) \ll (\log |G|)^C,
\]
where $C$ and the implied constant are absolute constants.
%\end{equation}
\end{conj}
(See \S \ref{subs:monmouth} for definitions of asymptotic notation.)

What about finite, simple, {\em abelian} groups $G$? They are the groups
$G=\mathbb{Z}/p\mathbb{Z}$. In that case, diameters can be very large:
for instance, $\diam \Gamma(\mathbb{Z}/p\mathbb{Z},\{1\}) = p-1$.
In general, when $G$ is abelian, the question
of which subsets $A\subset \mathbb{Z}/p\mathbb{Z}$ satisfy $|A+A|> K |A|$
for given $K$ is classical, and difficult; for $K$ a constant, it is answered
by a suitable generalization of Freiman's theorem \cite{MR2302736}.
(Freiman had done the case $G=\mathbb{Z}$; see \cite{MR0360496}, or
the exposition \cite{zbMATH01315258}.) The strongest result on the abelian case
to date is that of Sanders (\cite{Sanders}; based in part on
\cite{MR2738997}).

The Classification of Finite Simple Groups\footnote{Famed in mathematical lore
  as the theorem whose proof would be of the size of a
  large encyclopedia, were it all in one place.} tells us that all
finite, simple, non-abelian groups $G$ fall into three classes:
\begin{enumerate}
\item simple groups of Lie type, that is, matrix groups over finite fields
  (such as $\PSL_n(\mathbb{F}_q)$ or $\PSp_{2n}(\mathbb{F}_q)$), including some
  generalizations (twisted groups);
\item alternating groups $\Alt(n)$. The simple group $\Alt(n)$ is the unique
  subgroup of index $2$ of the group $\Sym(n)$ of all permutations of $n$ elements;
\item\label{it:spac}
  a finite list of exceptions, including, for example, the ``monster group''.
  \end{enumerate}
We can put (\ref{it:spac}) out of our minds, since it has a finite number of
elements, and we are aiming for asymptotic statements.

\subsubsection{Simple groups of Lie type (and bounded rank)}
  Our main goal in these notes will be to prove the following theorem.
  \begin{theorem}\label{thm:main}
    Let $G=\SL_2(K)$ or $G=\PSL_2(K)$, $K$ a field. Let
    $A\subset G$ be a set of generators of $G$.  Then either
\begin{equation}\label{eq:croissa}|A^3|\geq |A|^{1+\delta}\end{equation}
or \begin{equation}A^3 = G,\end{equation}
where $\delta>0$ is an absolute constant.
  \end{theorem}
  Here $\PSL_2(K) = \SL_2(K)/\{I,-I\}$, where $\SL_2(K)$ is, of course,
  the group of $2$-by-$2$ matrices with entries in a field $K$ and
  determinant $1$.
  The group $\PSL_2(K)$ is simple for $K=\mathbb{F}_q$ finite. 
  %reference?
  It is a group of Lie type; indeed, it will be our white mouse,
  in that it is convenient to work with, but sufficiently
  complex to be a good example of a large class.
  
Theorem \ref{thm:main}  was first proved in \cite{Hel08}
for $K=\mathbb{F}_p$, with $A^k=G$ ($k$ a constant) instead of $A^3=G$.
It then underwent a series of generalizations
(\cite{BGSU2}, \cite{MR2788087}, \cite{HeSL3}, \cite{GH1},
\cite{BGT} and \cite{MR3402696}, among others). By now, we know it
for every simple group of Lie type of bounded rank
(\cite{BGT}, \cite{MR3402696}).
The ``bounded rank'' condition means simply that
the constant $\delta$ in the inequality $|A^3|\geq |A|^{1+\delta}$ depends
on the rank of the group. (The rank of $\SL_n$ is $n-1$, that of
$\SO_n$ is $\lfloor n/2\rfloor$, etc.) In fact, there are examples
(due to Pyber) that show that $\delta$ has to depend on the rank.

We will give a proof of Thm.~\ref{thm:main} that descends from, but is not the
same as, the proof in \cite{Hel08}; it has strong influences from
\cite{HeSL3}, \cite{BGT} and \cite{MR3402696}.
In particular, the proof we shall give generalizes readily to
$\SL_n$ and other higher-rank groups; many of our intermediate results will be
stated for $\SL_n$, and the ideas carry over to other group families.

\begin{prob}\label{prob:agant}
   Let $K$ be a finite field.
   Let $G = \PSL_2(K)$ or $G=\SL_2(K)$. Let $S\subset G$ generate $G$.
   Using Thm.~\ref{thm:main}, prove that the diameter of $\Gamma(G,S)$
   is $\ll (\log |G|)^C$, where $C$ and the implied constant are absolute.
   Indeed, $C=O(1/\delta)$, where $\delta$ is the absolute constant in
   (\ref{eq:croissa}). Hint: apply Thm.~\ref{thm:main} repeatedly, with
   $S$ equal to 
   $A$, $A^3$, $A^9$,\dots 
\end{prob}
%It follows immediately from Thm.~\ref{thm:main} that, for any set $A$
%of generators of $G =\PSL_2(\mathbb{F}_q)$,
%\begin{equation}\label{eq:rudo}\diam \Gamma(G,A) \ll (\log |G|)^C,\end{equation}
%where $C=1/\delta$ and the implied constant is absolute.
In other words,
Babai's conjecture holds for $G = \PSL_2(\mathbb{F}_q)$. The bound
$\diam \Gamma(G,A)\ll (\log |G|)^C$ also holds for all
other simple groups of Lie type, only then $C$ depends
on the rank, since $\delta$ does.

Before \cite{Hel08}, $\Gamma(G,A)$ was known to be an expander for some particular sets of generators $A$ of $G=\SL_2(\mathbb{F}_q)$. In those cases, then,
the diameter bound $\diam \Gamma(G,A)\ll \log |G|$ was also known.
The main element of the proof came from modular forms (Selberg's spectral gap
\cite{MR0182610}).

Impatient readers may now jump to the body of the text and leave the
rest of the introduction for later. They should certainly read
\S \ref{sec:exprawsie}, on applications of Theorem \ref{thm:main}
to expander graphs.

\subsubsection{The simple group $\Alt(n)$}

For $G=\Alt(n)$, we have a statement that is somewhat weaker than
Babai's conjecture.

\begin{theorem}$($Helfgott-Seress, \cite{MR3152942}$)$\label{thm:hs}
Let $G=\Sym(n)$ or $G=\Alt(n)$. Let $A\subset G$ be a set of generators
of $G$.  Then
\begin{equation}\label{eq:hs}
  \diam(G,A) = e^{O\left( (\log n)^4 (\log \log n)\right)} =
e^{O_\epsilon\left((\log \log |G|)^{4+\epsilon}\right)}\end{equation}
for $\epsilon>0$ arbitrary.
\end{theorem}
In fact, the bound $\diam(G,A) = \exp(O((\log n)^4 (\log \log n)))$ holds
for all transitive groups $G<\Sym(n)$, and can be deduced from
Thm.~\ref{thm:hs}. 
We could state this result as follows:
let us be given a permutation puzzle with $n$ pieces that has a solution and
satisfies transitivity (that is, any piece can be sent to any other one
by some succession of moves). Then
there is always a short solution, starting from any reachable position.
Incidentally, non-transitive 
puzzles, such as Rubik's cube, can be
reduced to transitive ones at some cost, by means of Schreier generators.

We cannot have a product theorem just like Thm.~\ref{thm:main}
in $\Alt(n)$ or $\Sym(n)$. 
\begin{coun}[Pyber, Spiga]
  Let $H$ be the subgroup of $\Sym(n)$ consisting of all permutations of
  $\{1,\dotsc,m\}$. Let $\sigma$ be the cycle taking $i$ to $i+1$
  ($i\leq n-1$) and $n$ to $1$. Let $A = H \cup \{\sigma,\sigma^{-1}\}$. Then
\[|A^3| = \left|\{\sigma,\sigma^{-1},e\} \cdot H \cdot \{\sigma,\sigma^{-1},e\}
\cup H \sigma^{\pm 1} H\right| \leq 9 m! + 2 (m+1)! \leq (2 m + 11) |A|.\]
The factor $(2 m + 11)$ compared to $|A|$ for $A$ large; 
if we set, say, $m\sim n/2$, then $(2 m + 11) \ll |A|^{3/n}$.
\end{coun}

It might be that one of $|A^{O(n^C)}|\geq |A|^{1+\delta}$ or $A^{O(n^C)} = G$
always holds. Even having one of
$\left|A^{O(n^C)}\right| \geq |A|^{1+\delta/\log n}$ or $A^{O(n^C)} = G$ would be a definite
improvement over Thm.~\ref{thm:hs}. The exponents $4$ in
(\ref{eq:hs}) would become $3$, and, at any rate, as we shall later see,
product theorems have consequences other than diameter bounds.

It would be natural to hope that some ideas in \ref{thm:hs},
or its later version \cite{MR3931419},
or future strengthenings thereof, will be useful in addressing Babai's
conjecture over groups of Lie type of unbounded rank. It is not just
that the known counterexamples to strong product theorems over $\Sym(n)$ and
$\SL_n$ are related. There are ways to define the
``field with one element'' $\mathbb{F}_{\un}$, and objects over it;
then one generally obtains that $\Sym(n)\sim \SL_n(\mathbb{F}_{\un})$.
See, e.g., \cite{LorscheidFun}.

\subsubsection{Solvable and nilpotent groups}
A group $G$ is {\em solvable} if it has a subnormal series
\begin{equation}\label{eq:subno}
  \{e\}=H_0 \triangleleft H_1 \triangleleft \dotsb \triangleleft H_k = G
\end{equation}
all of whose quotients $H_{i+1}/H_i$ are abelian. As we said before,
questions on growth behave well under quotients, but such a reduction
does not help us as much as we would like,
since the best results available for the
abelian case are considerably less strong than
$|A\cdot A\cdot A|\geq |A|^{1+\delta}$.

A solvable group is {\em nilpotent} if it has a subnormal series
(\ref{eq:subno}) with $G_{i+1}/G_i$ contained in the center of
$G/G_i$ for every $0\leq i<k$. Nilpotent groups can often be seen as
``almost abelian'', and our context is no exception. One should not
hope to get stronger results on growth in nilpotent groups than for
abelian groups -- and, on the positive side, one can study nilpotent
groups with Freiman's and Ruzsa's tools, supplemented by a
Lie-algebra framework (\cite{MR3254927}; see also
\cite{MR2587441} and Tao \cite{MR2791295}).

What one can aim for is to show that, given a set $A$ in a solvable group,
either $A$ grows rapidly, or we are really in a nilpotent case.
We can make such a statement precise as follows.

\begin{conj}\label{conj:helflind}
Let $A\subset \GL_n(K)$,
$K$ a field. Assume that the group $\langle A\rangle$ generated by $A$
is solvable. Then, for any $C\geq 1$, either
\begin{equation}\label{eq:rodorar}
  |A^3|\geq C |A|\end{equation}
or there are subgroups $N\triangleleft G_0\triangleleft \langle A\rangle$
such that $G_0/N$ is nilpotent and
\begin{equation}\label{eq:rodorar2}
  N\subset (A\cup A^{-1}\cup \{e\})^k,\;\;\;\;\;
  \left|(A\cup A^{-1}\cup \{e\})^k \cap G_0\right|\geq C^{-O_n(1)} |A|,
\end{equation}
where $k$ depends only on $n$.
\end{conj}
We can, of course, set $C = |A|^\delta$, so
that (\ref{eq:rodorar}) has the familiar form $|A^3|\geq |A|^{1+\delta}$.

Gill and Helfgott proved Conjecture \ref{conj:helflind} for $K = \mathbb{F}_p$
\cite{MR3263161}.
The case $K=\mathbb{F}_q$ remains open. The case $K=\mathbb{C}$ is
relatively straightforward \cite{BrGrII}; in that case, the group $N$
can be taken to be trivial.

Putting the result for $K = \mathbb{F}_p$ together with
\cite{MR3402696}, it is simple to show that the same result holds for
$A\subset \GL_n(\mathbb{F}_p)$ general, without the assumption that
the group $\langle A\rangle$ generated by $A$
be solvable. (What \cite{MR3402696} does is reduce the general case to the solvable
case.) Again, the same conclusion is believed to hold over $\mathbb{F}_q$.
Breuillard, Green and Tao have proved \cite{BGTgen} that, if one is willing to
replace $C^{-O_n(1)}$ in (\ref{eq:rodorar2}) by a factor dependent
in an unspecified way on $C$ (but still independent of $|A|$), one
does not even need to assume that $A$ is contained in $\GL_n(K)$; they start
from a completely general, abstract group. They kindly gave the name
{\em Helfgott-Lindenstrauss conjecture} to the statement they proved, though
I would personally give that name to Conj.~\ref{conj:helflind}.
  
We shall study what is arguably the simplest interesting solvable case,
namely, the {\em affine group}
\begin{equation}\label{eq:boreloid2}
  \left\{\left(\begin{matrix} r & x\\0 &
      1\end{matrix}\right) : r\in K^*, x\in K\right\}.
\end{equation}
over a field $K$. As we shall see, the question of growth in it is essentially
equivalent to the {\em sum-product theorem} over a field.
Indeed, our treatment (\S \ref{subs:affi}) will show how to take
one of the ideas of proofs of the sum-product theorem over finite fields
(as in \cite{MR2053599} or \cite{MR2225493}) and reinterpret it in the context of groups (``pivoting''). A version of the same idea (really just a form of
induction) will appear again in our treatment of $\SL_2(K)$.

\subsubsection{Groups over $\mathbb{R}$ or $\mathbb{C}$}

The proof we shall give of Theorem \ref{thm:main} also works for $K$ infinite.
Even the first proof worked for $K=\mathbb{R}$, indeed more easily than over
$\mathbb{Z}/p\mathbb{Z}$. Actually the statement of Theorem \ref{thm:main}
turns out to have already been known over $\mathbb{R}$: the proof of
\cite[Thm.~2]{MR1869409} suffices to establish it.

Some results in combinatorics -- such as the sum-product theorem, which
underlay the first proof \cite{Hel08} of Thm.~\ref{thm:main}, or Beck's
theorem \cite{Beck}, on which \cite{MR1869409} relies --
are both stronger and easier to prove over the reals than over finite fields.
In fact, some results are known only over $\mathbb{R}$, or were known only
over $\mathbb{R}$ for many years. The reason is that, over $\mathbb{R}$,
the topology of the real plane can be used in the solution of geometrical
problems. A line divides the real plane into two halves; such a statement
does not hold or even make sense over $\mathbb{Z}/p\mathbb{Z}$.

As it turns out, for many applications, we need to know not just a statement
such as Theorem \ref{thm:main} for a linear group over the reals, but
a stronger version thereof. To be precise: one needs to show that the
maximal number $n_\delta(A)$ of points in $A$ separated by $\delta$ in the
real or complex metric grows: $n_\delta(A^3) \geq n_\delta(A)^{1+\delta}$.

Fortunately, as Bourgain and Gamburd first made
clear \cite{BGSU2}, existing proofs of Theorem \ref{thm:main} and its
generalizations can be modified to yield such stronger variants.
They worked with the proof in \cite{Hel08}, but the same should hold
of later proofs. The applications they found consisted in or involved
expander graphs. We will discuss results on expander graphs
in \S \ref{sec:exprawsie}.

%  the best general bound for the diameter known to date is
%  quasipolynomial \cite{MR3152942}, and thus not quite as good, qualitatively
%  speaking, as the bound we will prove for $\SL_2$. Further developments
%  in (b) may hold the key to good bounds on the diameter for Cayley
%  graphs of $\SL_n$ when $n\to \infty$. Perhaps surprisingly, we do have
%  essentially algorithmic results that work for most sets of generators
%  \cite{BBS04}, \cite{MR3272386}. The main reason may be that stochastic arguments
%  play a much larger role in the study of permutation groups than in the
%  study of groups of Lie type to date, perhaps precisely because, in
%  the case of permutation groups, there does not seem much else to use:
%  it seems hard to state problems in permutation groups in terms of algebraic
%  geometry.

 % For the same reason, alternating groups, and permutation
 % groups in general, lie outside the purview of these notes.
 % The reader is referred to the last part of the survey \cite{MR3348442}.
 % Let us just finish by saying that several of the ideas here -- in particular,
 % those having to do with orbits and induction -- are also useful in that
 % context.

\subsection{Notation}\label{subs:monmouth}
By $f(n)\ll g(n)$, $g(n)\gg f(n)$ and $f(n) = O(g(n))$ we mean the 
same thing, namely, that there are $N>0$, $C>0$ such that $|f(n)|\leq C\cdot 
g(n)$ for all $n\geq N$. We write $\ll_a$, $\gg_a$, $O_a$ if 
$N$ and $C$ depend on $a$ (say). 

As usual, $f(n) = o(g(n))$ means that
$|f(n)|/g(n)$ tends to $0$ as $n\to \infty$.
We write $O^*(x)$ to mean any quantity at most $x$ in absolute value.
Thus, if $f(n)=O^*(g(n))$, then $f(n)=O(g(n))$ (with $N=1$ and $C=1$).

Given a subset $A\subset X$, we let
 $1_A:X\to \mathbb{C}$ be the characteristic function of
$A$:
\[1_A(x) = \begin{cases} 1 &\text{if $x\in A$,}\\ 0 &\text{otherwise.}
\end{cases}
\]
%Given a set $A$ of elements of a group $G$, we write $\langle A\rangle$
%for the group generated by $A$.

\section{Elementary tools}
\subsection{Additive combinatorics}
Some of additive combinatorics can be described as the study of {\em sets
that grow slowly}. In abelian groups, results are often stated so as
to classify sets $A$ such that $|A^2|$ is not much larger than $|A|$;
in non-abelian groups, works starting with \cite{Hel08} classify sets
$A$ such that $|A^3|$ is not much larger than $|A|$. Why?

In an abelian
group, if $|A^2| < K |A|$, then $|A^k| < K^{O(k)} |A|$ -- i.e., if a set
does not grow after one multiplication with itself, it will not grow
under several.
This is a result of Pl\"unnecke \cite{MR0266892} and Ruzsa \cite{MR2314377}.
(Petridis \cite{MR3063158} recently gave a purely additive-combinatorial
proof.)

In a non-abelian group $G$, there can be sets $A$ breaking this rule.
\begin{prob}
  Let $G$ be a group. Let $H<G$, $g\in G\setminus H$ and
  $A = H \cup \{g\}$. Then $|A^2| < 3 |A|$, but $A^3\supset H g H$, and
$H g H$ may be much larger than $A$. Give an example with
  $G=\SL_2(\mathbb{F}_p)$. {\em Hint: let $H$ is the subgroup of
    $G$ consisting of the elements $g\in G$
    leaving the basis vector $e_1 = (1,0)$ fixed.}
%Error en survgr.tex!
\end{prob}

However, Ruzsa's ideas do carry over to the non-abelian case, as was
pointed out in \cite{Hel08} and \cite{MR2501249}.
%; in fact, \cite{MR810596}
%carries over without change, since the assumption that $G$ is abelian
%is never really used.
We must 
assume that $|A^3|$ is small, not just $|A^2|$, and then it does follow
that $|A^k|$ is small. The formal statement is Exercise \ref{exer:triplema},
below. To prove it, we need the following lemma.

\begin{lem}[Ruzsa triangle inequality]\label{lem:schatte}
Let $A$, $B$ and $C$ be finite subsets of a group $G$. Then
\begin{equation}\label{eq:eolt}
|A C^{-1}| |B| \leq |A B^{-1}| |B C^{-1}| .\end{equation}
\end{lem}
Commutativity is not needed. In fact, what is being used is in some
sense more basic than a group structure; as shown in \cite{MR3348800}, the same
argument works naturally in any abstract projective plane endowed with the
little Desargues axiom.
\begin{proof}
We will construct an injection
$\iota:A C^{-1} \times B \hookrightarrow A B^{-1} \times
 B C^{-1}$. For every 
$d\in A C^{-1}$, choose $(f_1(d),f_2(d)) = (a,c)\in A\times C$ such that
$d = a c^{-1}$. Define $\iota(d,b) = (f_1(d) b^{-1}, b (f_2(d))^{-1})$.
We can recover $d = f_1(d) (f_2(d))^{-1}$ from $\iota(d,b)$; hence
we can recover $(f_1,f_2)(d)=(a,c)$, and thus $b$ as well. Therefore,
 $\iota$ is an injection.
\end{proof}

\begin{prob}\label{exer:triplema}
  Let $G$ be a group. Prove that
\begin{equation}\label{eq:mony}
  \frac {|(A \cup A^{-1} \cup \{e\})^3|}{|A|} \leq \left(3\frac{|A^3|}{|A|}\right)^3\end{equation}
  for every finite subset $A$ of $G$. Show as well that, if $A = A^{-1}$
  (i.e., if $g^{-1}\in A$ for every $g\in A$), then
  \begin{equation}\label{eq:jotor}
    \frac{|A^k|}{|A|} \leq \left(\frac{|A^3|}{|A|}\right)^{k-2}.\end{equation}
  for every $k\geq 3$. Conclude that
  \begin{equation}\label{eq:marmundo}\frac{|(A\cup A^{-1}\cup \{e\})^k|}{|A|} \leq 3^{k-2} \left(\frac{|A^3|}{|A|}\right)^{3(k-2)}\end{equation}
  for every $A\subset G$ and every $k\geq 3$.
\end{prob}
Inequalities (\ref{eq:mony})--(\ref{eq:marmundo}) go back to
Ruzsa (or Ruzsa-Turj\'anyi \cite{MR810596}), 
at least for $G$ abelian.

This means that, from now on,
we can generally focus on studying when $|A^3|$ is or isn't much larger
than $|A|$. Thanks to (\ref{eq:mony}), we can also assume in many
contexts that $e\in A$ and $A = A^{-1}$ without loss of generality.

\subsection{The orbit-stabilizer theorem for sets}
A theme recurs in work on growth in groups: results on subgroups
can often be generalized to subsets. This is especially the case if the
proofs are quantitative, constructive, or, as we shall later see,
probabilistic.

The {\em orbit-stabilizer theorem for sets} is a good example, both because
of its simplicity (it should really be called a lemma) and because
it underlies a surprising number of other results on growth. 
It also helps to
put forward a case for seeing group actions, rather than groups themselves,
as the main object of study. 

We recall that an {\em action} $G\curvearrowright X$ is a homomorphism from a group $G$ to 
the group of automorphisms of a set $X$. (The automorphisms of a set $X$
are just the bijections from $X$ to $X$; we will see actions on objects with
richer structures later.)  For $A\subset G$ and $x\in X$, 
the {\em orbit} $A x$ is the
set $A x = \{g\cdot x : g\in A\}$. The {\em stabilizer} $\Stab(x)\subset G$ 
is given by $\Stab(x) = \{g\in G: g\cdot x = x\}$.

The statement we are about to give is
as in \cite[\S 3.1]{MR3152942}.

\begin{lem}[Orbit-stabilizer theorem for sets]\label{lem:orbsta}
Let $G$ be a group acting on a set $X$. Let $x\in X$, and let $A\subseteq G$ be non-empty. 
Then
\begin{equation}\label{eq:applepie}
|(A^{-1} A) \cap \Stab(x)|\geq \frac{|A|}{|A x|}.\end{equation}
Moreover, for every $B\subseteq G$,
\begin{equation}\label{eq:easypie}
|B A| \geq |A \cap \Stab(x)| |B x| .
\end{equation}
\end{lem}
The usual orbit-stabilizer theorem -- usually taught as part of a
first course in group theory -- states that, for $H$ a subgroup of $G$,
\[|H\cap \Stab(x)| = \frac{|H|}{|H x|}.\]
This the special case $A = B = H$ of the Lemma we (or rather you)
are about to prove.

\begin{prob}
  Prove Lemma \ref{lem:orbsta}. Suggestion: for
  (\ref{eq:applepie}), use the pigeonhole principle.
\end{prob}

If we try to apply Lemma \ref{lem:orbsta} to the (left)
action of the group $G$ on itself by left multiplication
\[g \mapsto (h\mapsto g\cdot h)\]
or to the (left) action by right multiplication
\[g \mapsto (h\mapsto h\cdot g^{-1}),\]
we do not get anything interesting: the stabilizer of any element is trivial.
The same is of course true of the right actions
$g\mapsto (h\mapsto g^{-1} h)$ and $g\mapsto (h\mapsto h \dot g)$.
However, we also have the action by conjugation
\[g \mapsto (h\mapsto g h g^{-1}).\]
The stabilizer of a point $h\in G$ is its {\em centralizer}
\[C(h) = C_G(h) = \{g\in G : g h = h g\};\]
the orbit of a point $h\in G$ under the action of the group $G$
is the {\em conjugacy class}
\[\Cl(h) = \{g h g^{-1}: g\in G\}.\]

Thus, we obtain the following result, which will show itself to be crucial
later. Its importance resides in making upper bounds 
on intersections of $A$ (or rather $A^{l+2}$)
with $\Cl(g)$ imply lower bounds on intersections of $A^2$ with
$C(g)$. In other words, the plan is to show that there are not too many
elements of $A^{l+2}$
of a special form, and then Lemma \ref{lem:lawve} will imply that there
are many elements of $A^{2}$ of another special form. Having many elements
of a special form will be very useful.

\begin{lem}\label{lem:lawve}
 Let $A\subset G$ be a non-empty set with $A = A^{-1}$. Then,
for every $g\in A^l$, $l\geq 1$,
\[|A^2\cap C(g)|\geq \frac{|A|}{|A^{l+2}\cap \Cl(g)|}.\]
\end{lem}
\begin{proof}
Let $G\curvearrowright G$ be the action of $G$ on itself by conjugation. Apply 
(\ref{eq:applepie}) with $x=g$; the orbit of $g$ under conjugation by $A$ 
is contained in $A^{l+2}\cap \Cl(g)$.
\end{proof}

It is instructive to see some other consequences of Lemma \ref{lem:orbsta}.
\begin{prob}
Let $G$ be a group and $H$ a subgroup thereof. Let $A\subset G$ be a 
 set with $A=A^{-1}$. Then
\begin{equation}\label{eq:vento}
|A^2 \cap H| \geq \frac{|A|}{r},\end{equation}
where $r$ is the number of cosets of $H$ intersecting $A$. 
\end{prob}
{\em Hint:} Consider the action $G\curvearrowright X=G/H$ by left multiplication, that is,
$g\mapsto (a H \mapsto g a H)$. Then apply (\ref{eq:applepie}).

\vskip 5pt
The following exercise
tells us that, if we show that the intersection of $A$ with
a subgroup $H$ grows rapidly, then we know that $A$ itself grows rapidly.
\begin{prob}\label{prob:bara}
  Let $G$ be a group and $H$ a subgroup thereof. 
Let $A\subset G$ be 
a non-empty set with $A = A^{-1}$.
 Prove that, for any $k>0$,
\begin{equation}\label{eq:avoc1}
|A^{k+1}| \geq \frac{|A^{k}\cap H|}{|A^2\cap H|} |A| .
\end{equation}
\end{prob}
    {\em Hint:} Consider the action $G\curvearrowright G/H$ again, and
    apply both (\ref{eq:easypie}) and (\ref{eq:applepie}).

\begin{prob}\label{prob:bere}
Let $G$ be a group and $H$ a subgroup thereof.  Write $\pi_{G/H}:G\to G/H$
for the quotient map.
Let $A\subseteq G$ be a non-empty set with $A = A^{-1}$. Then, for any $k>0$,
\[|A^{k+2}| \geq \frac{|\pi_{G/H}(A^k)|}{|\pi_{G/H}(A)|} |A| .\]
\end{prob}

\section{Growth in a solvable group}

\subsection{Remarks on abelian groups}
Let $G$ be an abelian group and $A$ be a finite subset of $G$. This
is the classical setup for what nowadays is called
{\em additive combinatorics} -- a field that may be said to have started
to split off from additive number theory with Roth \cite{MR0051853}
and Freiman \cite{MR0360496}.

In general, for $G$ abelian, $A\subset G$ may be such that $|A+A|$
is barely larger than $|A|$, and that is the case even if we assume
that $A$ generates $G$. For instance, take $A$ to be a segment
of an arithmetic progression: $A = \{2,5,8,\dotsc,3m-1\}$. Then
$|A| = m$ and $|A+A| = 2m-1 < 2 |A|$.

Freiman's theorem \cite{MR0360496} (generalized first to
abelian groups of bounded torsion \cite{MR1701207} and then to arbitrary
abelian groups \cite{MR2302736}) tells us that, in
a very general sense, this is the only kind of set that grows slowly.
We have to start by giving a generalization of what we just called
a segment of an arithmetic progression.

\begin{defn}
  Let $G$ be a group.
  A {\em centered convex progression} of dimension $d$
  is a set $P\subset G$ such that there exist
  \begin{enumerate}
  \item a convex subset $Q\subset \mathbb{R}^d$ that is also symmetric
    ($Q = -Q$),
  \item a homomorphism $\phi:\mathbb{Z}^d\to G$,
  \end{enumerate}
  for which $\phi(\mathbb{Z}^d\cap Q) = P$. We say $P$ is {\em proper}
  if $\phi|_{\mathbb{Z}^d\cap Q}$ is injective.
\end{defn}

\begin{prop}[Freiman; Ruzsa-Green]
  Let $G$ be an abelian group. Let $A\subset G$ be finite. Assume that
  $|A+A|\leq K |A|$ for some $K$. Then $A$ is contained in at most
  $c_{K,1}$ copies of $P+H$ for some proper, centered convex progression $P$
  of dimension $\leq c_{K,2}$ and some finite subgroup $H<G$ such that
  $|P+H| \ll e^{c_{K,2}} |A|$. Here $c_{K,1},c_{K,2}>0$ depend only on $K$.
\end{prop}
The best known bounds are essentially
those of Sanders \cite{Sanders}, as improved
by Konyagin (see \cite{MR2994996}): $c_{K,1}, c_{K_2}\ll (\log K)^{3+o(1)}$.

This is a broad field into which we will not venture further. Notice just
that, in spite of more than forty years of progress, we do not yet have
what is conjectured to be the optimal result, namely, the above with
$f(K), g(K)\ll \log K$ (the ``polynomial Freiman-Ruzsa conjecture'').
Thus the state of our knowledge here is in some sense
less satisfactory than in the case of simple groups, as will later become
clear.

The situation for nilpotent groups is much like the situation for
abelian groups: there is a generalization of the Freiman-Ruzsa theorem
to the nilpotent case, due to Tointon \cite{MR3254927} (see also
Tessera-Tointon \cite{ToinTess}), based on groundwork
laid by Fisher-Katz-Peng
\cite{MR2587441} and Tao \cite{MR2791295}.

{\em Brief excursus.} There is of course also the matter of the role of nilpotent groups in the study
of growth in a different if related sense, within geometric group theory: for 
$A$ a subset of an infinite group $G$, how does $|A^k|$ behave as $k\to\infty$?
It is easy to see that, if $G$ is nilpotent, then $|A^k|$ grows polynomially on $k$.
Gromov's theorem \cite{MR623534}, a deep and celebrated result, states
the converse: if $|A^k|$ is bounded by a polynomial on $k$, then $\langle A\rangle$
has a nilpotent subgroup of finite index. There are several clearly distinct proofs
of Gromov's theorem by now; of them, the one closest to the study of ``growth'' in the sense of the present paper is clearly \cite{MR2833482}. See \cite{BGTgen} for
further work in that direction.
%mention Tessera-Tointon and the fact that we wish for an analogue
%of Sanders' bounds! Contact Matthew?

%Given a set of generators $A$ of $\mathbb{Z}$, it is trivial to give
%a very fast
%algorithm that expresses any given $m\in \mathbb{Z}$ as a word (i.e.,
%a product of elements of $A$ and their inverses) of length
%$O(|n| + |m|/|n|)$, where $n$ is the element of $A$ whose absolute value
%is largest. The general case does not seem much harder -- essentially
%because one can use induction on the normal series (\ref{eq:turu}).
%\begin{prob}
 % Let $A=\{a_1,a_2\}$ or $A=\{a_1,a_2,a_3\}$
%  be a set of generators of the Heisenberg group $H(K)$
% (\ref{eq:heisenb}) with $K = \mathbb{Z}/p\mathbb{Z}$. Our task,
%  given any element $g$ of $H(K)$, is to
%  find a word of length $O(p^{3/2}) = O(\sqrt{|H(K)|})$ on $A$ equal to $g$.
%  Show that this can be done in time polynomial on $\log p$.
% (Note that inverting an element of $(\mathbb{Z}/p\mathbb{Z})^*$ takes
%  time linear on $\log p$, by the Euclidean algorithm.)
%\end{prob}

\subsection{The affine group}\label{subs:affi}
\subsubsection{Growth in the affine group}
We defined the {\em affine group} $G$ over a field $K$
in (\ref{eq:boreloid2}).
(If we were to insist on using language in exactly the same way as later,
we would say that the affine group is an algebraic group $G$ (a variety
with morphisms defining the group operations) and that (\ref{eq:boreloid2})
describes the group $G(K)$ consisting of its rational points. For the sake
of simplicity, we avoid this sort of distinction here. We will go over
most of these terms once the time to use them has come.)

Consider the following subgroups of $G$:
\begin{equation}\label{eq:gutr}
U = \left\{\left(\begin{matrix} 1 & a\\0 &1\end{matrix}\right) :
a\in K\right\},\;\;\;\;\;\;
T = \left\{\left(\begin{matrix} r & 0\\0 &1\end{matrix}\right) :
r\in K^*\right\}.\end{equation}
These are simple examples of a {\em solvable} group $G$, of a maximal
{\em unipotent} subgroup $U$ and of a maximal torus $T$. 
In general, in $\SL_n$,
a maximal torus is just the group of matrices that are diagonal with
respect to some fixed basis of $\overline{K}^n$, or, what is the same,
the centralizer of any element that has $n$ distinct eigenvalues.
Here, in our group $G$, the centralizer $C(g)$ of any element $g$
of $G$ not in $U$ is a maximal torus.

When we are looking at what elements of the group $G$ do to each other
by the group operation, we are actually looking at two actions:
that of $U$ on itself (by the group operation)
and that of $T$ on $U$ (by conjugation; $U$ is a normal subgroup of $G$). They turn out to correspond to
addition and multiplication in $K$, respectively:
\[\begin{aligned}
\left(\begin{matrix} 1 & a_1\\0 &1\end{matrix}\right) \cdot
\left(\begin{matrix} 1 & a_2\\0 &1\end{matrix}\right) &=
\left(\begin{matrix} 1 & a_1+a_2\\0 &1\end{matrix}\right)\\
\left(\begin{matrix} r & 0\\0 &1\end{matrix}\right) \cdot
\left(\begin{matrix} 1 & a\\0 &1\end{matrix}\right) \cdot
\left(\begin{matrix} r^{-1} & 0\\0 &1\end{matrix}\right) 
&= \left(\begin{matrix} 1 & r a\\0 &1\end{matrix}\right) .
 \end{aligned}\]

Thus, we see that growth in $U$ under the actions of $U$ and $T$ 
is tightly linked to growth in $K$ under addition and 
multiplication. This can be seen as motivation for studying growth in
the affine group $G$. Perhaps we need no such motivation: we are studying
growth in general, through a series of examples, and the affine group
is arguably the simplest interesting example of a solvable group.

At the same time, the study of growth in a field under addition and
multiplication was historically important in the passage from the
study of problems in commutative groups (additive combinatorics) to the
study of problems in noncommutative groups by related tools. (Growth
in noncommutative groups had of course
been studied before, but from very different perspectives, e.g.,
that of geometric group theory.) Some of the ideas we are about to see
in the context of groups come ultimately from \cite{MR2053599} and
\cite{MR2359478}, which are about finite fields, not about groups.

Of course, the way we choose to develop matters emphasizes what the
approach to the affine group has in common with the approach to other, not
necessarily solvable groups. The idea of {\em pivoting} will appear again
when we study $\SL_2$.

\begin{lem}\label{lem:cano}
Let $G$ be the affine group over $\mathbb{F}_p$. Let $U$
be the maximal unipotent subgroup of $G$, and $\pi:G\to G/U$ the quotient
map.

Let $A\subset G$, $A=A^{-1}$.
Assume $A\not\subset U$; let $x$ be an element of $A$ not in $U$.
 Then 
\begin{equation}\label{eq:hastar}
|A^2 \cap U| \geq \frac{|A|}{|\pi(A)|},\;\;\;\;\;\;
|A^2 \cap T| \geq \frac{|A|}{|A^5|} |\pi(A)|
\end{equation}
for $T=C(x)$.
\end{lem}
Recall $U$ is given by (\ref{eq:gutr}). Since $x\not\in U$, its centralizer
$T=C(x)$ is a maximal torus.
\begin{proof}
By (\ref{eq:vento}),
$A_u:= A^2 \cap U$ has at least $|A|/|\pi(A)|$ elements.
Consider the action of $G$ on itself by conjugation. Then, by Lemma 
\ref{lem:orbsta}, $|A^2\cap T|\geq |A|/|A(x)|$. 
(Here $A(x)$ is the orbit of $x$ under the action of $A$ by conjugation,
and $\Stab(x) = C(g) = T$ is the stabilizer of $g$ under conjugation.) We set
$A_t := A^2\cap T$. Clearly, 
$|A (x)| = |A(x) x^{-1}|$. Since the derived group of $G$ is $U$
(meaning, in particular, that $a x a^{-1} x^{-1}\in U$ for any $a$ and $x$),
we see that $A(x) x^{-1} \subset A^4\cap U$, and so
$|A(x)|\leq |A^4 \cap U|$.  At the same time, by (\ref{eq:easypie})
applied to the action $G\curvearrowright G/U$ by left multiplication, 
$|A^5| = |A^4 A|\geq |A^4 \cap U|\cdot |\pi(A)|$. Hence 
\[|A_t| \geq \frac{|A|}{|A^4\cap U|} \geq \frac{|A|}{|A^5|} |\pi(A)|.\] 
\end{proof}

The proof of the following proposition will proceed essentially by induction.
This may be a little unexpected, since we are in a group $G$, not in,
say, $\mathbb{Z}$, which has a natural ordering. However, as the proof will
make clear, one can do induction on a group with a finite set of generators,
even in the absence of an ordering.

\begin{prop}\label{prop:jutor}
Let $G$ be the affine group over $\mathbb{F}_p$, $U$ the maximal unipotent
subgroup of $G$, and $T$ a maximal torus. Let $A_u\subset U$, $A_t\subset T$. 
Assume
$A_u = A_u^{-1}$, $e\in A_t, A_u$ and $A_u\ne \{e\}$. Then
\begin{equation}\label{eq:jces}|(A^2_t(A_u))^6|
\geq \min(|A_u| |A_t|,p).\end{equation}
\end{prop}
To be clear: here
\[A_t^2(A_u) = 
\{t_1(u_1): t_1\in A_t^2, u_1\in A_u\},\]
where $t(u) = t u t^{-1}$, since $T$ acts on $U$ by conjugation.
\begin{proof}
Call $a\in U$ a {\em pivot}
if the function $\phi_{a}:A_u \times A_t \to U$ given by
\[(u,t)\mapsto u t(a) =  u t a t^{-1}\]
is injective. 
%(Here $t(a)$ refers to the element to which $t$ takes $a$
%by conjugation, i.e., $t a t^{-1}$.)

{\em Case (a): There is a pivot $a$ in $A_u$.} 
Then
$|\phi_a(A_u,A_t)| = |A_u| |A_t|$, and so 
\[|A_u A_t(a)| \geq |\phi_a(A_u,A_t)| = |A_u| |A_t| .\]
This is the motivation for the name ``pivot'': the element $a$ is the pivot
on which we build an injection $\phi_a$, giving us the growth we want.

{\em Case (b): There are no pivots in $U$.} 
%Then, for every $a\in U$,
%there are $u_1,u_2\in A_u$ and $t_1,t_2\in A_t$ such that $\phi_a(u_1,t_1) = 
%\phi_a(u_2,t_2)$.
As we are about to see, 
this case can arise only if either $A_u$ or $A_t$ is large with respect to $p$.
Say that $(u_1,t_1), (u_2,t_2)\in A_u\times A_t$ {\em collide} for $a\in U$ if
$\phi_a(u_1,t_1) = \phi_a(u_2,t_2)$. Saying that there are no pivots in $U$
is the same as saying that, for every $a\in U$, there are at least two distinct
$(u_1,t_1), (u_2,t_2)\in A_u\times A_t$ that collide for $a$. Now, two distinct
$(u_1,t_1)$, $(u_2,t_2)$ can collide for at most one $a\in U\setminus \{e\}$.
(As one can easily see, such an $a$ corresponds to a solution
to a non-trivial linear equation, which can have at most one solution.)
Hence, if there are no pivots, $|A_u|^2 |A_t|^2 \geq |U\setminus \{e\}| = p-1$,
i.e.,
$|A_u|\cdot |A_t|$ is large ($\geq \sqrt{p-1}$).
This fact already hints that this case will not be hard.

Let $\kappa_a$ denote the number of collisions for a given $a\in U$:
\[\kappa_a = |\{u_1,u_2\in A_u, t_1,t_2\in A_t: \phi_a(u_1,t_1) = 
\phi_a(u_2,t_2)\}|.\]
As we were saying, 
two distinct $(u_1,t_1)$, $(u_2,t_2)$ collide for at most
one $a\in U\setminus \{e\}$. Hence
the total number of collisions $\sum_{a\in U\setminus \{e\}} \kappa_a$
is $\leq |A_u|^2 |A_t|^2$, and so there is an $a\in U\setminus \{e\}$ such that
\[\kappa_a \leq \frac{|A_u|^2 |A_t|^2}{p-1} .\]
Now,
\[\begin{aligned}(|A_u| |A_t|)^2 &=
\left(\sum_{x\in \phi_a(A_u,A_t)} |\{(u,t)\in A_u\times A_t : \phi_a(u,t)=x\}|\right)^2 \\ &\leq |\phi_a(A_u,A_t)|
\sum_{x\in \phi_a(A_u,A_t)} |\{(u,t)\in A_u\times A_t : \phi_a(u,t)=x\}|^2
\\ &= |\phi_a(A_u,A_t)|\cdot \kappa_a,\end{aligned}\]
where the inequality is just Cauchy-Schwarz.
Thus, $|\phi_a(A_u,A_t)|\geq |A_u|^2 |A_t|^2/\kappa_a$,
and so
\[|\phi_a(A_u,A_t)|\geq \frac{|A_u|^2 |A_t|^2 }{
\frac{|A_u|^2 |A_t|^2}{p-1}}
= p-1.\]

We are not quite done, since $a$ may not be in $A$.
Since $a$ is {\em not} a pivot (as there are none), there exist distinct
$(u_1,t_1)$, $(u_2,t_2)$ such that $\phi_a(u_1,t_1)=\phi_a(u_2,t_2)$. Then
$t_1\ne t_2$ (why?), and so the map $\psi_{t_1,t_2}:U\to U$ given by
$u\mapsto t_1(u) (t_2(u))^{-1}$ is injective. The idea is that the very
non-injectivity of $\phi_a$ gives an implicit definition of it, much like
a line that passes through two distinct points is defined by them.

What follows may be thought
of as the ``unfolding'' step, in that we wish to remove an element $a$
from an expression, and we do so by applying to the expression a map that will
send $a$ to something known. We will be using the commutativity of $T$ here.

For any $u\in U$, $t\in T$, since $T$ is abelian,
\begin{equation}\label{eq:unfold}\begin{aligned}
\psi_{t_1,t_2}(\phi_a(u,t)) &= t_1(u t(a)) (t_2(u t(a)))^{-1} =
t_1(u) t(t_1(a) (t_2(a))^{-1}) (t_2(u))^{-1}\\
&= t_1(u) t(\psi_{t_1,t_2}(a)) (t_2(u))^{-1} = t_1(u) t(u_1^{-1} u_2) (t_2(u))^{-1}
,\end{aligned}\end{equation}
where $\psi_{t_1,t_2}(a) = u_1^{-1} u_2$ holds because $\phi_a(u_1,t_1) =
\phi_a(u_2,t_2)$. Note that $a$ has 
disappeared from the last expression in (\ref{eq:unfold}). We obtain 
\[\psi_{t_1,t_2}(\phi_a(A_u,A_t))\subset A_t(A_u) A_t(A_u^2) A_t(A_u)
\subset (A_t(A_u))^4.\] Since
$\psi_{t_1,t_2}$ is injective, we conclude that
\[|(A_t(A_u))^4| \geq |\psi_{t_1,t_2}(\phi_a(A_u,A_t))| = 
|\phi_a(A_u,A_t)| \geq p-1,\]
that is to say, at most a single element of $U$ is missing from
$(A_t(A_u))^4$. Since $A_u$ contains at least one element besides $e$,
we obtain immediately that
\[(A_t(A_u))^6 \supset (A_t(A_u))^4 A_u = U.\]

There is an idea here that we are about to see again: any element $a$ that 
is not a pivot can, by this very fact, 
be given in terms of some $u_1, u_2\in A_u$, 
$t_1, t_2\in A_t$, and so an expression involving $a$ can often be transformed
into one involving only elements of $A_u$ and $A_t$.

{\em Case (c): There are pivots and non-pivots in $U$.} 
Here comes what we can think of as the inductive step.
Since $A_u \ne \{e\}$, $A_u$ generates
$U$. Thus, there is a non-pivot $a\in U$ and a $g\in A_u$ such that
$g a$ is a pivot. Then $\phi_{ag}:A_u\times A_t\to U$ is injective. Much as 
in (\ref{eq:unfold}), we unfold:
\begin{equation}\begin{aligned}
\psi_{t_1,t_2}(\phi_{g a}(u,t)) &= t_1(u t(g) t(a))
(t_2(u t(g) t(a)))^{-1} \\
&= 
t_1(u t(g)) t(u_1^{-1} u_2) (t_2(u t(g)))^{-1} ,
\end{aligned}\end{equation}
where $(u_1,t_1)$, $(u_2,t_2)$ are distinct pairs such that 
$\phi_a(u_1,t_1) = \phi_a(u_2,t_2)$. Just as before,
$\psi_{t_1,t_2}$ is injective.
Hence
\[|A_t(A_u) A_t^2(A_u) A_t(A_u^2) A_t^2(A_u) A_t(A_u)| \geq 
|\psi_{t_1,t_2}(\phi_{g a}(u,t))| = |A_u| |A_t|\]
and we are done.

The idea to recall here is that, if $S$ is a subset of an orbit $\mathscr{O} = \langle A\rangle x$ such that $S\ne \emptyset$ and $S\ne \mathscr{O}$, then
there is an $s\in S$ and a $g\in A$ such that $g s\not\in S$. 
%We use the point at which we escape from $S$.
It is in this fashion that we
can use induction even in the absence of a natural ordering of $\langle A\rangle$.
\end{proof} 
 We are using the fact that $G$ is the affine group over $\mathbb{F}_p$
(and not over some other field) only at the beginning of case (c), when we
say that, for $A_u\subset U$, $A_u \ne \{e\}$ implies 
$\langle A_u\rangle = U$.

\begin{prop}\label{prop:sabat}
Let $G$ be the affine group over $\mathbb{F}_p$.
Let $U$
be the maximal unipotent subgroup of $G$, and $\pi:G\to G/U$ the quotient
map.

Let $A\subset G$, $A=A^{-1}$, $e\in A$. Assume $A$ is not contained in any
maximal torus.
Then either
\begin{equation}\label{eq:jomar}
|A^{73}| \geq \sqrt{|\pi(A)|} \cdot |A|
\end{equation}
or 
\begin{equation}
U \subset A^{72}
.\end{equation}
\end{prop}
The exponents $72$, $73$ in (\ref{eq:jomar}) are not optimal. For instance,
one can obtain $52$, $53$ by looking closer at the proof of Prop.~\ref{prop:jutor}.
\begin{proof}
We can assume $A\not\subset U$, as otherwise what we are trying to prove
is trivial. Let $g$ be an element of $A$ not in $U$; its centralizer $C(g)$
is a maximal torus $T$. By assumption, there is an element $h$ of $A$ 
not in $T$. Then $h g h^{-1} g^{-1}\ne e$. At the same time, $h g h^{-1} g^{-1}$
does lie in
$A^4\cap U$, and so $A^4\cap U$ is not $\{e\}$.

Let $A_u = A^4 \cap U$, $A_t=A^2\cap T$; their size is bounded 
from below by (\ref{eq:hastar}). 
Applying Prop.~\ref{prop:jutor}, we obtain
\[|A^{72}\cap U| \geq \min(|A_u| |A_t|,p)  \geq
\min\left(\frac{|A|^2}{|A^5|},p\right).\]
By (\ref{eq:easypie}), $|A^{73}|\geq |A^{72}\cap U|\cdot |\pi(A)|$. Clearly,
if $|A|/|A^5| < 1/\sqrt{|\pi(A)|}$, then $|A^{57}| \geq |A^5| > \sqrt{|\pi(A)|}
\cdot |A|$. If $|A|/|A^5| \geq 1/\sqrt{|\pi(A)|}$, then either
$|A^{72}\cap U| \geq |A|/\sqrt{|\pi(A)|}$ and so
$|A^{73}|\geq \sqrt{|\pi(A)|}\cdot |A|$, or $|A^{72}\cap U| = p$ and
so $U\subset A^{72}$.
\end{proof}
%, but, qualitatively 
%speaking, Prop.\ \ref{prop:sabat} is as good a result as one can aim to for now.
For $A\subset U$, getting a better-than-trivial lower bound on $|A^{k}|$,
$k$ a constant, amounts to
Freiman's theorem in $\mathbb{F}_p$, and getting a growth
factor of the form $|\pi(A)|^{\delta}$, $\delta>0$, would involve
proving a version of Freiman's theorem of polynomial strength. As we discussed
before, that is a difficult open problem.

\subsubsection{Brief remarks on a generalization and an application}

We can see Prop.\ \ref{prop:sabat} as a very simple result of the
``classification of approximate subgroups'' kind.
If a set $A$ (with $A=A^{-1}$, $e\in A$)
in the affine group over $\mathbb{F}_p$ grows slowly
($|A^k| \leq |A|^{1+\delta}$, $k=73$, $\delta$ small) then either
(i) $A$ is contained in a maximal torus,
(ii) $A$ is contained in a few cosets of the maximal unipotent subgroup $U$
(that is, $|\pi(A)|\leq |A|^{2\delta}$), or
(iii)
  $A^k$ contains a subgroup (namely, $U$) such that
  $\langle A\rangle/H$ is nilpotent (here, in fact, abelian).
%What we have just done, then, is to prove the simplest case of what
%\cite{BGTstru} calls the ``Helfgott-Lindenstrauss conjecture''. That conjecture
%%states, in essence, that one can give a classification of slowly growing
%$A$ like the one above when $A$ is a subset of any linear group. A
%qualitative version of the conjecture
%was proven in \cite{BGTstru}; this means, in practice,
%that one can say something about the case in which $|A^k|$ is at most
%a constant times $|A|$, but we cannot quite
%yet say something in the full general case when $|A^k|$ is just assumed
%to be at most $|A|^{1+\delta}$. For a proof for linear groups over
%$\mathbb{F}_p$, see \cite{MR3263161}. There is clearly
%work that remains to be done here.

\begin{prob}
  Give examples of subsets $A$ of the affine group over $\mathbb{F}_p$ that
  fail to grow for each of the reasons above: a set contained in a maximal
  torus, a set almost contained in $U$, and a set containing $U$.
\end{prob}

The following more general statement has been proved for $K=\mathbb{F}_p$
\cite{GH1}. (It remains open for general finite $K$.)
Let $A\subset G = \GL_n(K)$ ($A=A^{-1}$, $e\in A$)
be such that $\langle A\rangle$ is solvable.
Then, for any $\delta>0$, if $|A^3|<|A|^{1+\delta}$, there are a
subgroup $S\triangleleft \langle A\rangle$ and a unipotent subgroup
$U\triangleleft S$ such that (a) $S/U$ is nilpotent,
(b) $U\subset A^k$, where $k=O_n(1)$, (c) $A$ is contained in
$|A|^{O_n(\delta)}$ cosets of $U$. 
%best we can hope for: Freiman

\begin{prob}
  Verify that each of the cases (i)-(iii) enumerated above in the case
  of the affine group satisfies this description, i.e., there are $S$ and $U$
  such that (a)--(c) are fulfilled.
\end{prob}

What is also interesting is that the results we have proved
on growth in the affine linear
group can be interpreted as a {\em sum-product theorem}.

\begin{prob}\label{prop:dudor}
  Let $X\subset \mathbb{F}_p$, $Y\subset \mathbb{F}_p^*$ be given with
  $X = -X$, $0\in X$, $1\in Y$. Using Prop.~\ref{prop:jutor}, show that
  \begin{equation}\label{eq:huru}|6 Y^2 X|\geq
 \min(|X| |Y|,p-1).\end{equation}
\end{prob}
This is almost exactly \cite{MR2359478}, Corollary 3.5], say.

  Using (\ref{eq:huru}), or any estimate like it, one can prove the following.
\begin{theorem}[Sum-product theorem \cite{MR2053599},
\cite{MR2225493}; see also \cite{MR1948103}]\label{thm:orb}
For any $A\subset \mathbb{F}_p^*$
with $|A| \leq p^{1 - \epsilon}$, $\epsilon>0$, we have
\[\max(|A\cdot A|,|A+A|) \geq |A|^{1 + \delta},\]
where $\delta>0$ depends only on $\epsilon$.
\end{theorem} 
In fact, the proof we have given of Prop.~\ref{prop:jutor} takes its ideas
from proofs of the sum-product theorem. In particular, the idea of
{\em pivoting} is already present in them. We will later see how to
apply it in a broader context.

\subsubsection{Diameter bounds in a remaining case}

We have proved that growth occurs in $\SL_2$ under some weak
conditions. This leaves open the question of what happens with $A^k$,
$k$ unbounded, for $A$ not obeying those conditions. In particular:
what happens when $A$, while not contained in the maximal unipotent
group $U$, is contained in the union of few cosets of $U$?

One thing that is certainly relevant here is that, in general,
there is no  vertex
  expansion in the affine group, and thus no expansion. Indeed,
the purpose of this subsection is to give a glimpse of the issue of diameter
bounds in situations in which neither expansion nor rapid growth hold.

Let us state the lack of vertex expansion in elementary terms.
\begin{prop}\label{prop:garn}
  For any $\lambda_1,\dotsc,\lambda_k\in \mathbb{Z}$, and any
  $\epsilon>0$ , there is a constant $C$ depending on $\epsilon$
  such that, for every
  prime $p>C$, there is a set $S\subset \mathbb{F}_p$,
  $0<|S|\leq p/2$, such that
  \begin{equation}\label{eq:kuklo}
    |S\cup (S+1) \cup \lambda_1 S \cup \dotsc \cup \lambda_k S|\leq
  (1+\epsilon) |S|.\end{equation}
\end{prop}

\begin{prob}\label{prob:whar}
  Prove Proposition \ref{prop:garn}. {\em Hints:} prove this for $k=1$ first;
  you can assume $\lambda = \lambda_1$ is $\geq 2$. Here is a plan. 
  We want to show that $|S\cup (S+1) \cup \lambda S|\leq (1+\epsilon) |S|$.
  For $|S\cup (S+1)|$ to be $\leq (1+\epsilon/2) |S|$, it is enough that
  $S$ be a union of intervals of length $> 2/\epsilon$.
  (By an {\em interval} we mean the image of an interval $\lbrack a,b\rbrack
  \cap \mathbb{Z}$ under the map $\mathbb{Z}\to \mathbb{Z}/p\mathbb{Z}
  \sim \mathbb{F}_p$.)
  We also want
  $|S\cup \lambda S|\leq (1+\epsilon) |S|$; this will be the case 
  if $S$ is the union of disjoint sets of the form $V$, $\lambda^{-1} V$,
  \dots, $\lambda^{-r} V$, $r\geq \epsilon/2$. Now, in $\mathbb{F}_p$, if $I$
  is an interval of length $\ell$, then $\lambda^{-1} I$ is the union
  of $\lambda$ intervals (why? of what length?). Choose $V$ so that
  $V, \lambda^{-1} V, \dotsc, \lambda^{-r} V$ are disjoint. Let $S$
  be the union of these sets; verify that it fulfills (\ref{eq:kuklo}).
\end{prob}
The following exercise shows that Prop.~\ref{prop:garn} is closely connected
to the fact that a certain group is {\em amenable}.
\begin{prob}
  Let $\lambda\geq 2$ be an integer. Define the {\em Baumslag-Solitar group}
  $\BS(1,\lambda)$ by
  \[\BS(1,\lambda) = \langle a_1,a_2| a_1 a_2 a_1^{-1} = a_2^{\lambda}\rangle.\]
  \begin{enumerate}
  \item A group $G$ with generators $a_1,\dotsc,a_\ell$ is called
    {\em amenable} if, for every $\epsilon>0$, there is a finite $S\subset G$
    such that \[|F\cup a_1 F \cup \dotsc \cup a_\ell F|\leq (1+\epsilon)
    |F|.\]
    Show that $\BS(1,\lambda)$ is amenable. {\em Hint:} to construct $F$,
    take your inspiration from Exercise \ref{prob:whar}.
  \item Express the subgroup of the affine group over $\mathbb{F}_p$
    generated by the set
    \begin{equation}\label{eq:anfesio}
      A_\lambda=\left\{\left(\begin{matrix}\lambda & 0\\0 & 1\end{matrix}\right),
      \left(\begin{matrix}1 & 1\\0 & 1\end{matrix}\right)\right\}\end{equation}
        as a quotient of $\BS(1,\lambda)$, i.e., as the image of a homomorphism
        $\pi_p$ defined on $\BS(1,\lambda)$.
  \item Displace or otherwise modify your sets $F$ so that, for each
    of them, $\pi_p|_F$ is injective for $p$ larger than a constant.
    Conclude that $S = \pi_p(F)$ satisfies (\ref{eq:kuklo}), thus giving
    a (slightly) different proof of exercise \ref{prob:whar}.
    \end{enumerate}
  \end{prob}
    
Amenability is not good news when we
are trying to prove that a diameter is small, in that it closes
a standard path towards showing that it is logarithmic in the size of the
group. However, it does not imply that the diameter is not small.

Let us first be clear about what we can prove or rather about what we
cannot hope to prove.
We should not aim at a bound on the diameter of the affine group $G$ with
respect to an arbitrary set of generators $A$:
it is easy to choose $A$ so that the diameter
of $\Gamma(G,A)$ is very large.

\begin{prob}\label{prob:saignant}
  Let $A_\lambda$ be as in (\ref{eq:anfesio}) for 
  $\lambda$ a generator of $\mathbb{F}_p^*$. Let
  $A = A_\lambda \cup A_\lambda^{-1}$.
  Then $A$ generates the affine
  group $G$ over $\mathbb{F}_p$.
  Show that $\diam \Gamma(G,A) = (p-1)/2$.
\end{prob}

Rather, we should aim for a bound on the diameter of the {\em Schreier
  graph} of the action of the affine group $G$ by conjugation on its maximal
unipotent subgroup $U$. In general, the Schreier graph of an action
$G\curvearrowright X$ of a group $G$ on a set $X$ with respect to a set
of generators $A$ of $G$ is the graph having $X$ as its set of vertices
and $\{(x,a x): x\in X, a\in A\}$ as its set of edges.
In our case ($X=U$, $A=A_\lambda \cup A_\lambda^{-1}$,
$\lambda\in \mathbb{F}_p^*$),
the Schreier graph is isomorphic to the graph $\Gamma_{p,\lambda}$
      with vertex set $\mathbb{F}_p$ and edge set
      \[\{(x,x+1): x\in \mathbb{F}_p\} \cup
      \{(x,x-1): x\in \mathbb{F}_p\} \cup
      \{(x,\lambda x): x\in \mathbb{F}_p\} \cup
      \{(x,\lambda^{-1} x): x\in \mathbb{F}_p\} .\]
      We are not avoiding the problem posited by the fact that
the Baumslag-Solitar group $\BS(1,\lambda)$ is amenable, since what amenability impedes is precisely
      a natural approach
      to prove logarithmic diameter bounds on $\Gamma_{p,\lambda}$.
      If Proposition \ref{prop:garn} were not true, then the diameter
      of $\Gamma_{p,\lambda}$ would be $O(\log p)$. (Why?)

      If $\lambda$ is the projection of a fixed integer $\lambda_0$, then
      it is possible, and easy, to give a logarithmic diameter bound
      nevertheless.
      \begin{prob}
        Let $\lambda_0\geq 2$ be an integer. Let $\lambda = \lambda_0 \mo p$,
        which lies in $\mathbb{F}_p^*$ for $p>\lambda_0$.
        Show that the diameter of the graph $\Gamma_{p,\lambda}$ 
        is $O(\lambda_0 \log p)$. {\em Hint:} lift elements of
        $\mathbb{F}_p$ to $\mathbb{Z}\cap \lbrack 0,p-1\rbrack$, and write them out in base
        $\lambda_0$.
      \end{prob}

      It turns out to be possible to give a {\em polylogarithmic} bound for
      general $\lambda\in \mathbb{F}_p^*$:
      \begin{equation}\label{eq:dadojo}
        \diam \Gamma_{p,\lambda} \ll (\log p)^{O(1)},
      \end{equation}
      where the implied constants are independent of $p$ and
      $\lambda$. Here we need not assume that
      $\lambda$ generates $\mathbb{F}_p^*$, but we do assume
      that the order of $\lambda$ is $\gg \log p$.
      (Indeed, if the order of $\lambda$ is very small, viz.,
      $o((\log p)/\log \log p)$, then (\ref{eq:dadojo}) cannot hold; why?)     

      The proof of (\ref{eq:dadojo}) was the outcome of
      a series of discussions among B. Bukh, A. Harper, E. Lindenstrauss
      and the author. It is essentially an exercise in Fourier analysis
      using bounds on exponential sums due to Konyagin \cite{MR1289921}.
      \begin{prob}
Let $p$ be a prime, $\lambda\in \mathbb{F}_p^*$. Assume $\lambda$ has
order $\geq \log p$.
 Write $e(t) = e^{2\pi i t}$ and $e_p(t) = e^{2\pi i t/p}$. Konyagin \cite[Lemma 6]{MR1289921} showed that, for any $\epsilon>0$, there is a $c_\epsilon>0$ such that,
for any $p\geq c_\epsilon$ prime
and $\alpha, \lambda\in (\mathbb{Z}/p\mathbb{Z})^*$ with $\lambda$ of order
$\geq c_\epsilon (\log p)/(\log \log p)^{1-\epsilon}$ in the group
$(\mathbb{Z}/p\mathbb{Z})^*$,
\begin{equation}\label{eq:dolto}
  \sum_{j=0}^J |\{\alpha \lambda^j/p\}|^2\geq \frac{1}{(\log p)^{3^{\epsilon/4}}},
  \end{equation} 
where  $J = \lfloor c_\epsilon \log p (\log \log p)^4\rfloor$ and
$\{x\}$ is the element of $(-1/2,1/2\rbrack$ such that
$x-\{x\}$ is an integer.
\begin{enumerate}
  \item
Show that (\ref{eq:dolto}) implies that $S(\alpha) = \sum_{j=0}^J e_p(\alpha \lambda^j)$
satisfies $|S(\alpha)|\leq J + 1 - 1/(\log p)^{3^{\epsilon/4}/2}$ for
every $\alpha\in (\mathbb{Z}/p\mathbb{Z})^*$.
\item
Deduce that every element of $\mathbb{Z}/p\mathbb{Z}$ can
be written as a sum $\sum_{i=1}^K \lambda^{j_i}$,
where $0\leq j_i\leq J$ and $K$ is bounded by
\[K\ll J (\log p)^{3^{\epsilon/4}/2} (\log p)
\ll_\epsilon (\log p)^{2+3^{\epsilon/4}/2} (\log \log p)^4
\ll_\epsilon (\log p)^{5/2 + \epsilon}.\]
To do so, show first
that for any sequence $r_0,\dotsc,r_j\in \mathbb{Z}/p\mathbb{Z}$,
the number of ways of expressing $x\in \mathbb{Z}/p\mathbb{Z}$ as a sum
of $K$ elements (not necessarily distinct) of a subset
$A\subset \mathbb{Z}/p\mathbb{Z}$ equals
\[\frac{1}{p} \sum_{\alpha\in \mathbb{Z}/p\mathbb{Z}}
S_A(\alpha)^K e_p(- \alpha x),\]
where $S_A(\alpha) = \sum_{a\in A} e(\alpha a)$. This approach is 
the {\em circle method} over $\mathbb{Z}/p\mathbb{Z}$.
\item
Conclude that
the graph $\Gamma_{p,\lambda}$
      with vertex set $\mathbb{F}_p$ and edge set
      \[\{(x,x+1): x\in \mathbb{F}_p\} \cup
      \{(x,\lambda x): x\in \mathbb{F}_p\}\]
      has diameter $\ll_\epsilon (\log p)^{5/2+\epsilon}$.
      \end{enumerate}
\end{prob}
%      ;
%      the actual mathematical content lies mainly in those bounds. Going there would take us too far afield, but it is only a good thing if
%      the reader's
%      interest in techniques rather different from the ones discussed here
%      has been whetted.

      %The proof suggested by the hint is actually constructive: given
      %$\lambda_0$, a prime $p$ and a vertex
      %$x\in \mathbb{F}_p$, it constructs a path of length
      %$O(\lambda_0 \log p)$ from the origin $0$ to $x$. In other words,
      %we know how to navigate in the graph.

      %Can we forget about $\lambda_0$,
      %work with $\lambda\in \mathbb{F}_p^*$ arbitrary, and give
      %a good bound that is independent of $\lambda$? This is the subject
     %of one of the course projects.

%      At the time of writing of these notes, 
%      a bound of the quality $O((\log p)^{O(1)})$ is known
%      (unpublished, based on work by Konyagin), but we have no efficient
%      algorithm yet for navigating $\Gamma_{p,\lambda}$ in time $O((\log p)^{O(1)})$.
%      We assume here the condition that
%      the order of $\lambda$ in $\mathbb{F}_p^*$ is $\gg \log p$;
%      indeed, if the order of $\lambda$ is $o(\log p/\log \log p)$,
%      then the diameter of $\Gamma_{p,\lambda}$ is {\em not} $O((\log p)^{O(1)})$.
%      (Why?).
      
\section{Intersections with varieties}\label{chap:inter}

Let $G$ a linear algebraic group defined over a field $K$. Let $A$
be a finite set of generators of the set of points of $G$ over $K$.

We will first show that, 
unless all the points of $G$ over $K$ lie in $V$,
there are (plenty of) elements of $A^k$, $k$ bounded, that do not lie on $V$
({\em escape from subvarieties}). Here the constant $k$ depends only
on some invariants of $V$ (its number of components, their degree and
their dimension), not on $K$ or on other properties of $V$.

Our main aim will then be to show that, if $A$ grows slowly, then
$A$ is truly a beautiful object, very regular from many points of view.
Of course, this is a strategy for showing in the following section
that $A$ does not exist (or is almost all of $G$).

``Very regular'' here means ``behaving well with respect to the algebraic
geometry of the ambient group $G$''. To be precise: the intersection of
a slowly growing set $A$ with any variety $V$ will be bounded by
not much more than $|A|^{\dim(V)/\dim(G)}$ (Theorem \ref{thm:lp}; the
{\em dimensional estimate}).

Here is an intuitive image. Thinking for a moment in three dimensions (that is, $\dim(G)=3$), one might say that this estimate
means that $A$ is very regular in the sense of being a roughly spherical blob, as its intersection with any
line, or any curve of bounded degree,
is bounded by $O(|A|^{1/3})$,
and its intersection with any plane, or any surface of bounded degree,
is bounded by $O(|A|^{2/3})$.

Finally, we will see that for some kinds of varieties $V$
-- namely, centralizers --
we can give a {\em lower} bound on the intersection of $A$ with $V$,
roughly of the same order as the upper bound above. This fact will be
a crucial tool in \S \ref{sec:growthth}.

\subsection{Preliminaries from algebraic geometry and algebraic groups}\label{subs:prelim}

We will have the choice of working sometimes over linear algebraic groups
and sometimes
over Lie algebras (as in \cite{MR3348442}, following \cite{HeSL3})
or solely over linear algebraic groups (as in \cite{MR3309986},
which follows \cite{BGT}). We will follow the first path. Naturally, we will
 need some preliminaries on varieties, 
their behavior under mappings, the derivatives of such mappings, and so forth.
It will all be a quick review for some readers. When it comes to
basic algebraic geometry,
we will cite mainly \cite{MR1748380} and \cite{Hartshorne}, as they are
standard sources for English speakers. In the case of either source, we will limit ourselves
to the first chapter, that is, to classical foundations. Our definitions for
terms related to
algebraic groups come mostly from \cite{zbMATH01219612} and \cite{zbMATH00050185}; basic facts on finite groups of Lie type come from
\cite[ch.~21 and 24]{malle2011linear}.

%\footnote{Meaning that the canonical references in French and Russian are of course different.}
  .
\subsubsection{Basic definitions.}
We will need some basic terms from algebraic geometry.
Let $K$ be a field; denote by $\overline{K}$ an algebraic closure of $K$.
For us, a variety $V$ will simply be an affine or a projective variety --
that is, the algebraic set consisting of the solutions in $\mathbb{A}^n$
to a system of polynomial equations,
or the solutions in $\mathbb{P}^n$
to a system of homogeneous polynomial equations. 
We say $V$ is defined over $K$ if $V$
can be described by polynomial equations with coefficients in $K$.
%In particular, a variety may be reducible or irreducible.
%The coefficients of the equations are assumed to lie on a field $K$.
Given a field $L$ containing $K$, we write $V(L)$ for the set of solutions
with coordinates in $L$. When we simply say ``points on $V$'', we mean
elements of $V(\overline{K})$.

Abstract algebraic varieties
%(let alone schemes) 
%Of course, abstract algebraic varieties (defined to be locally like
%varieties in the above sense, together with other properties; see
(as in, say, \cite[Def.~I.6.2]{MR1748380}) will not really be needed, although they
do give a
very natural way to handle a variety that parametrizes a family
of varieties, among many other things. For
instance, we will tacitly refer to the variety of all $d$-dimensional planes in projective space, and, while that variety (a {\em Grassmanian})
can indeed be defined as an algebraic set in projective space, that is a
non-obvious though standard fact.

The {\em Zariski topology} on $\mathbb{A}^n$ or $\mathbb{P}^n$
is the topology whose open sets are the complements of varieties
(affine ones if we work in $\mathbb{A}^n$, projective ones if we work in
$\mathbb{P}^n$).
It induces a topology, also called Zariski topology,
on any variety $V$; its open sets are the complements $V\setminus W$
of subvarieties $W$ of $V$. (A {\em subvariety} of $V$ is a variety contained
in $V$.) The {\em Zariski closure} $\overline{S}$ of a subset $S$ of $V$ is its
closure in the Zariski topology.

A variety $V$ is {\em irreducible} if it is not the union of two varieties
$V_1, V_2 \ne \emptyset, V$. (Note that many authors call an algebraic set a variety
only if it is irreducible.)
Every variety $V$
can be written as a finite union of irreducible varieties $V_i$, with
$V_i \not\subset V_j$ for $i\ne j$; they are called the {\em irreducible
  components} (or simply the components) of $V$.

When we say
``property $P$ holds for a generic point in the variety $V$'',
we simply means that there is a dense open subset $U\subset V$
such that property $P$ holds for every point on $U$.
It is easy to see that a non-empty open subset of an irreducible variety
is always dense.

The {\em dimension} $\dim V$
of an irreducible variety $V$ is the largest $d$ such that
there exists a chain of irreducible varieties
\[V_0 \subset V_1 \subset \dotsb \subset V_d = V.\]
The union of several irreducible varieties of dimension $d$ is called
a {\em pure-dimensional} variety of dimension $d$. If $W$ is a
pure-dimensional proper subvariety
of an irreducible variety $V$, then $\dim W < \dim V$
\cite[Cor.~I.7.1]{MR1748380}. (A subvariety $W\subset V$ is {\em proper}
if $W\ne V$.) 

The direct product $V\times W$ of irreducible varieties $V$, $W$
is an irreducible variety of dimension is $\dim V + \dim W$
(\cite[Exer.~I.3.15 and I.2.14]{Hartshorne} or \cite[Prop.~I.6.1,
Thm.~I.6.3 and Prop.~I.7.5]{MR1748380}).

%We will not be using schemes. In particular,
%have no need to define the generic point of a variety in the abstract.

%direct product of irreducible vars. is irreducible (reference)

\subsubsection{Degrees. Bézout's theorem.}
The degree of a pure-dimensional
variety $V$ in $\mathbb{A}^n$ or $\mathbb{P}^n$ of dimension $d$ is its
number of points of intersection with a generic plane of dimension $n-d$.
(See? We just referred tacitly to\dots) 

{\em B\'ezout's theorem}, in its classical formulation, states that,
for any two distinct irreducible curves $C_1$, $C_2$ in $\mathbb{A}^2$,
the number of points of intersection $(C_1\cap C_2)(\overline{K})$ is
at most $d_1 d_2$. (In fact, for $C_1$ and $C_2$ generic, the number of
points of intersection is exactly $d_1 d_2$; the same is true
for {\em all} distinct $C_1$, $C_2$ if we count points of intersection
with multiplicity.)

In general, if $V_1$ and $V_2$ are irreducible varieties, and we write
$V_1\cap V_2$ as a union of irreducible varieties $W_1, W_2,\dotsc,W_k$
with $W_i\not\subset W_j$ for $i\ne j$, a generalization of B\'ezout's
theorem tells us that
\begin{equation}\label{eq:bezout}
  \sum_{i=1}^k \deg(W_k) \leq \deg(V_1) \deg(V_2).\end{equation}
See, for instance, 
\cite[p.251]{MR1658464}, where Fulton and MacPherson are mentioned in
connection to this and even more general statements.

Inequality (\ref{eq:bezout}) implies immediately
that, if a variety $V$ is defined
by at most $m$ equations of degree at most $d$, then the number and degrees
of the irreducible components of $V$ are bounded in terms of $m$ and $d$ alone.

\subsubsection{Morphisms.}\label{subsub:morph}

A {\em morphism} from a variety $V_1\subset \mathbb{A}^{m}$ to a variety
$V_2\subset \mathbb{A}^{n}$ is simply a map $f:V_1\to V_2$ of the form
\[(x_1,\dotsc,x_{m})\mapsto (P_1(x_1,\dotsc,x_{m}),\dotsc,
P_{n}(x_1,\dotsc,x_{m})),\] where $P_1,\dots,P_n$ are polynomials.
It is clear that the preimage $f^{-1}(W)$ of a subvariety $W\subset V_2$
is a subvariety of $V_1$.

What is not at
all evident a priori is that, for $W\subset V_1$ a subvariety, the image
$\phi(W)$ is a {\em constructible set}, meaning a finite union of
 terms of the form $W\setminus W'$,
where $W$ and $W'\subset W$ are varieties. (For instance, if
$V\subset \mathbb{A}^2$ is the variety given by $x_1 x_2 = 1$ (a hyperbola),
then its image under the morphism $\phi(x_1,x_2) = x_1$ is the constructible
set $\mathbb{A}^1 \setminus \{0\}$.) This result is due to Chevalley
\cite[Cor.~I.8.2]{MR1748380}.\footnote{As R. Vakil says of the closely
  related statement that the image of a projective variety under a morphism
  is a projective variety: ``a great deal of classical algebra and geometry
  is contained in this theorem as special cases.'' In model-theoretical
  terms, we are talking of quantifier elimination.}
%; it is closely related to the main result of
%classical elimination theory (\cite[]{} or \cite[\S I.9, Thm.~2]{MR1748380}).
  
%degree?

Let $V$ be irreducible and let $f:V\to \mathbb{A}^n$
be a morphism. It is easy to see that the Zariski closure $\overline{f(V)}$
must be irreducible, and that $\dim \overline{f(V)} \leq \dim V$.
Let $d =\dim V - \overline{f(V)}$.
Then there is
a Zariski open subset $U\subset \overline{f(V)}$ such that, for
every $x\in U$, the preimage $f^{-1}(\{x\})$ is a pure-dimensional variety
of dimension $d$ \cite[Thm.~I.8.3]{MR1748380}.

It is easy to see
(by B\'ezout (\ref{eq:bezout})) that
the degree of $f^{-1}(\{x\})$ is
bounded in terms of $\deg(V)$, $n$ and the degrees of the
polynomials $P_1,\dotsc,P_n$ defining $f$.
If $\dim V = \overline{f(V)}$, $f^{-1}(\{x\})$ is $0$-dimensional, and so
its number of points is bounded by its degree, by the definition of degree.

\subsubsection{Tangent spaces and derivatives}

Let $V\subset \mathbb{A}^n$ be a variety of dimension $d$ defined by
equations $P_i(x_1,\dotsc,x_n)=0$, $1\leq i\leq k$.
The {\em tangent space} $T_x V$ of $V$ at $x$ is the kernel
of the linear map from $\mathbb{A}^n$ to $\mathbb{A}^k$ given by the
matrix $\mathscr{P}|_x = \left(\partial P_i/\partial x_j\right)_{1\leq i\leq k, 1\leq j\leq n}$.
(These are formal partial derivatives.)
A point $x$ on $V$ is {\em non-singular} if $\dim T_x V = \dim V$, and
{\em singular} otherwise.
The set of singular points is a proper subvariety of $V$
\cite[Thm.~I.5.3]{Hartshorne}.

Let $V\subset \mathbb{A}^n$, $W\subset \mathbb{A}^m$ be varieties and
let $f:V\to W$ be a morphism. 
At any point $\vec{x}$ on $V$, the linear map given
by the matrix $J|_x = \left(\frac{\partial f_i}{\partial x_j}\right)_{1\leq i\leq m,1\leq j\leq n}$
restricts to a linear map $Df|_x:T_x V \to T_x W$ (as follows from
the chain rule). For any $r\geq 0$, the set of non-singular points on $V$
such that the rank of $Df|_x$ is at least $r$ is Zariski-open in $V$.
This fact is easy to see for $V=\mathbb{A}^n$: the rank is then $<r$ if and only
if every $r$-by-$r$ minor of $J|_x$ is $0$, a condition that defines a subvariety. For $V$ general, define a new matrix by putting the matrix $\mathscr{P}|_x$
on top of the matrix $J|_x$, and note that the new matrix will have rank at least
$n-\dim(V)+r$ if and only if $Df|_x$ has rank at least $r$; thus we
can proceed as for $V=\mathbb{A}^n$.

\begin{prob}
  Let $V$, $W$ be varieties, $V$ irreducible, $f:V\to W$ a morphism,
  and $x$ a non-singular point on $V$. Prove that, if the rank of
  $Df|_x$ is at least $r$, then the dimension of $\overline{f(V)}$ is at least
  $r$.
\end{prob}

\subsubsection{Linear algebraic groups} 

A {\em linear algebraic group} over a field $K$
is a subvariety $G$ of $\GL_n$, defined over $K$,
that is closed under multiplication and inversion.\footnote{Alternatively, we could define a linear algebraic group $G$ to be an affine variety with two morphisms
  $\cdot:G \times G\mapsto G$ and ${}^{-1}:G\to G$ satisfying the usual rules,
  and then prove that $G$ is isomorphic to a subvariety of $\GL_n$ with the
  multiplication and inversion morphisms it inherits from $\GL_n$
  \cite[Prop.~1.10]{zbMATH00050185}.}
We thus have morphisms $\cdot:G\times G\mapsto G$ and ${}^{-1}:G\to G$.
An {\em algebraic} or {\em closed} subgroup of $G$ is a subvariety $H$ of $G$
that is also closed under multiplication and inversion.

We will assume that the field of definition
$K$ is {\em perfect}, meaning that
every finite extension of $k$ is separable;
this assumption will save us from possible
trouble. Finite fields, fields of characteristic
$0$ and algebraically closed fields are always perfect fields.

A linear algebraic group $G$ is {\em semisimple} if it has no
connected,
%\footnote{Note that an algebraic group is connected (in the Zariski topology) if and only if it is irreducible 
 non-trivial and solvable normal algebraic subgroups, even defined over
 $\overline{K}$. (``Connected'' means ``connected in the Zariski topology;
 an algebraic group is connected if and only if it is irreducible
 \cite[Prop.~2.2.1]{zbMATH01219612}. For algebraic groups, being
      {\em solvable} is defined analogously as for groups
      \cite[\S 2.4]{zbMATH00050185}.)
We say $G$ is {\em simple} (over $K$) if it is semisimple,
connected and has no connected, proper and non-trivial normal algebraic subgroups defined over $K$.\footnote{Some sources (e.g., \cite[\S 22.8]{zbMATH00050185}) give the name {\em almost-simple} to what we call {\em simple}.}
%; (b)
%  since $K$ is perfect, 
%  we could replace the word ``semisimple'' by ``non-abelian''
%in the definition of a simple group
%  without changing the meaning \cite{}.}

Let $G$ be an arbitrary linear algebraic group over a field $K$.
An element $g\in G(\overline{K})$.
is {\em semisimple} if it is diagonalizable over $\overline{K}$.
Note that, by \cite[\S 4.3, Prop.]{zbMATH00050185} 
and the first definition in \cite[\S 4.5]{zbMATH00050185},
the semisimplicity of $g$ is invariant under isomorphisms of $G$, i.e.,
it does not actually depend on the embedding of $G$ into $\GL_n$.

A {\em torus} $T<\GL_n$
is an algebraic group isomorphic to $\GL_1^r$ over $\overline{K}$
for some $r\geq 1$.
A torus defined over $K$ is always diagonalizable over $\overline{K}$
\cite[\S 8.5, Prop.]{zbMATH00050185}; that is, there exists $g\in \GL_n(\overline{K})$ such that $g T g^{-1}$ is a subgroup of the group of diagonal matrices
in $\GL_n$.
A {\em maximal torus} of a connected linear algebraic group $G$ is a
torus $T<G$ with $r$ maximal. We call $r$ the {\em rank} of $G$.
If $G$ is connected, then
every semisimple $g\in G(\overline{K})$ lies in a maximal torus
\cite[Thm.~6.4.5(ii)]{zbMATH01219612}.

The centralizer $C(g)$ of a semisimple point $g$ in $G$ has dimension at least
$r=\rank(G)$; if $\dim C(g) = \rank(G)$, we say $g$ is {\em regular}. When
$G$ is semisimple, a semisimple element
$g\in G(\overline{K})$ is regular if and only if
the connected component $C(g)^\circ$
of $C(g)$ containing the identity is a maximal torus
(\cite[\S 12.2, Prop., and \S 13.17, Cor.~2(c)]{zbMATH00050185}).
A regular semisimple element $g\in G(\overline{K})$ lies in exactly one
maximal torus \cite[\S 12.2, Prop.]{zbMATH00050185}.
%Regular elements form a non-empty open subset of $G$
%\cite[Exer 6.4.15.2(a)]{zbMATH01219612}.
For $G$ semisimple,
regular semisimple elements form a non-empty open subset of $G$
\cite[\S 2.14]{steinberg1965regular}.

\subsubsection{Lie algebras}

A {\em Lie algebra} is a vector space $\mathfrak{g}$ over a field $K$
together with a bilinear map $\lbrack \cdot,\cdot\rbrack: \mathfrak{g}\times
\mathfrak{g}\to \mathfrak{g}$ satisfying the identities
\begin{equation}\label{eq:bracket}\lbrack x,y\rbrack = - \lbrack y,x\rbrack,\;\;\;\;\;\;
\lbrack x,\lbrack y,z \rbrack\rbrack +
\lbrack y,\lbrack z,x \rbrack\rbrack +
\lbrack z,\lbrack x,y \rbrack\rbrack = 0.\end{equation}
An {\em ideal} of a Lie algebra is a subspace $\mathfrak{v}\subsetneq
\mathfrak{g}$ such that $\lbrack g,\mathfrak{v}\rbrack \subset \mathfrak{v}$.
We say a Lie algebra is {\em simple} if it has no ideals other than $(0)$.

A linear algebraic group $G$ acts on its tangent space $\mathfrak{g}=T_e G$
at the origin
by conjugation: for $g\in G$, we
define the linear map $\Ad_g:\mathfrak{g}\to \mathfrak{g}$ to be the
derivative of $y\mapsto g y g^{-1}$. The derivative of $\Ad_g$ with respect
to $g$ can be written as a bilinear map
$\mathfrak{g}\times \mathfrak{g}\to \mathfrak{g}$, which we call
$\lbrack \cdot,\cdot\rbrack$; it is fairly straightforward to check that it
satisfies the identities in (\ref{eq:bracket}), and thus
makes $\mathfrak{g}$
into a Lie algebra.

It is easy to see that, if a subspace $\mathfrak{v}$ of the Lie algebra
$\mathfrak{g}$ of a linear algebraic group $G$ is invariant under $\Ad_g$
for every $g\in G$, then $\mathfrak{v}$ is an ideal. Thus, if $G$ is not simple,
then $\mathfrak{g}$ is not simple.

It would be convenient if $G$ simple implied $\mathfrak{g}$ simple,
but that is not quite true\footnote{To the contrary of what was carelessly
  stated in the proof of Prop.~5.3 in the survey \cite{MR3348442}.}.
However, there are only a few exceptions, all in small characteristic.
 To summarize: for $G=\SL_n$,
the Lie algebra $\mathfrak{g} = \mathfrak{sl}_n$ is simple provided that
the characteristic $p$ of the field $K$ does not divide $n$.
(If $p|n$, then
%The problem here
%comes from the diagonal matrix $I$, whose trace $n$ equals $0$ in $K$
%when $p|n$. It is a non-trivial element of the center of
%$\mathfrak{g}
$\mathfrak{sl}_n$ has non-trivial center, namely, the multiples of the
diagonal matrix $I$.) For almost simple Lie groups $G$ such that $\mathfrak{g}$
is not isomorphic to $\mathfrak{sl}_n$, we have that $\mathfrak{g}$
is simple provided that $\charac(K)>3$ \cite[Table 1]{zbMATH03807840}.
(The assumption in \cite{zbMATH03807840} that the ground field is
algebraically closed is harmless, as, if $\mathfrak{g}$ is simple over
$\overline{K}$, it follows trivially that $\mathfrak{g}$ is simple over
$K$: a decomposition over $K$ would also be valid over $\overline{K}$.)
In fact, $\charac(K)>2$ is enough for all Lie algebras of
type other than $A_n$ (corresponding to $\SL_n$), $E_6$ and $G_2$, by
the same table.

In spite of this small-characteristic phenomenon, we will nevertheless descend
from the algebraic groups to
Lie algebra at an important step (proof of Lemma \ref{lem:tush}), as then matters arguably become
particularly clear and straightforward.

%The Lie algebra $\mathfrak{g}$
%associated to a linear algebraic group $G$ is the tangent
%space $T|_{e} G$ to $G$ at the origin, with $\lbrack \cdot,\cdot\rbrack$
%defined to be the derivative (with respect to $x$) of the derivative
%(with respect to $y$) of the map $G\times G\to G$ given by
%$(x,y)\mapsto x y x^{-1} y^{-1}$.

%Cor 3.18(i) in Borel

\subsubsection{Finite groups of Lie type}

The general definition of a finite group of Lie type is that it is the
group $G^F$
of points on a semisimple algebraic group $G$ defined over a finite field
$\mathbb{F}_q$ that are left fixed by a {\em Steinberg endomorphism} $F:G\to G$.
A Steinberg endomorphism is an endomorphism $F:G\to G$ such that, for some
$m\geq 1$, $F^m$ is the {\em Frobenius map} with respect to $\mathbb{F}_q$.
The Frobenius map with respect to $\mathbb{F}_q$ is the map sending every
element $g\in G(\overline{\mathbb{F}_q})$ with entries $g_{i,j}$ to the element with
entries $g_{i,j}^q$. It fixes precisely the elements of $G(\mathbb{F}_q)$.

The most familiar finite groups of Lie type (classical groups and Chevalley groups) are
of the form $G(\mathbb{F}_q)$, $G$ a semisimple algebraic group; they correspond to the case $m=1$. The groups that require $m>1$ are called {\em twisted groups}.

We will work out growth in $G(K)$, $G=\SL_2$, $K$ finite (or, more generally,
perfect) in a way that generalizes easily to other groups of Lie type
with $G$ simple.
It is possible to include twisted groups, as was shown in \cite{MR3402696};
however, our notation will be of the form $G(K)$, as is appropriate for $m=1$.

Requiring $G$ to be simple is not quite the same as requiring
the group of Lie type $G^F=G(K)$ to be simple. The simple groups coming
from groups of Lie type are of the form $G^F/Z(G^F)$, $G$ simple.\footnote{Two
  comments for the sake of precision are in order. (a)
  There is one group in the classification of finite simple groups that
  is almost but not quite of the type $G^F/Z(G^F)$: the {\em Tits group}
  \cite[p.~213]{malle2011linear}. As we said before, we need not care about
  individual groups in the classification, since we aim at asymptotic
  statements. (b) By a result of Tits
  \cite[Thm.~24.17]{malle2011linear},
given $G$ simple and {\em simply connected} \cite[Def.~9.14]{malle2011linear},
the group $G^F/Z(G^F)$ will be simple, provided we are not in a finite list
of exceptions. Notably, 
$\SO_n$ is not simply connected; one uses a simply-connected
finite cover of $\SO_n$ in its stead.}
The center $Z(G^F)$ is described in \cite[Table~24.2]{malle2011linear}.
It is very easy to pass from statements
on growth in $G^F$ to statements on growth in $G^F/Z(G^F)$, as we will see in
the case for $G=\SL_2$, where $Z(G^F)=\{I,-I\}$.

%of Lie type 

%. However, by a result of Tits
%
%     endomorphism $F$, and letting $G^F$ denote the group of Lie type
%     consisting of the fixed points of $F$, the group $G^F/Z(G^F)$ will
%     be simple, except for a finite list of special cases.
% groups of Lie type

\subsection{Escape from subvarieties}
We are working with a finite subset $A$ of a group $G$.
At some points in the argument, we will need to make sure that we can
find an element $g\in A^k$ ($k$ small) that is {\em not} special: for example,
we want to be able to use a $g$ that is not unipotent, that does not have
a given $\vec{v}$ as an eigenvector, that {\em is} regular semisimple, etc. 

It is possible to give a completely general argument of this form. Let us
first set the framework. Let $G$ be a group acting by linear transformations
on $n$-dimensional space $\mathbb{A}^n$ over a field $K$.
In other words, we are given a homomorphism
$\phi:G\to \GL_n(K)$ from $G$ to the group of invertible matrices $\GL_n(K)$.
Let $W$ be a proper subvariety of $\mathbb{A}^n$.
We may think of points on $W$ as being {\em special},
and points outside $W$ as being generic.
We start with a point $x$ of
$\mathbb{A}^n$, and a subset $A$ of $G$. The following proposition ensures us
that, if, starting from $x$ and acting on it repeatedly by $A$, we can
eventually escape from $W$, then we can escape from it in a bounded number of
steps, and in many ways.

The proof\footnote{The statement of the proposition is as in \cite{HeSL3},
  based closely on \cite{MR2129706}, but the idea is probably older.}
proceeds by induction on the dimension, with the degree kept under control.
What is crucial for us is that the dimension is an integer, and thus
can be used as a counter for induction. (Alternatively, we could say that
the kind of induction we are about to undertake works because the ring
$K\lbrack x_1,\dotsc,x_n\rbrack$ is Noetherian.)

\begin{prop}\label{prop:huru}
  Let us be given
  \begin{itemize}
    \item $G$ a group acting linearly on affine space $\mathbb{A}^n$ over
      a field $K$,
    \item $W\subsetneq \mathbb{A}^n$, a subvariety,
    \item $A$ a set of generators of $G$ with $A =A^{-1}$, $e\in A$, 
\item $x\in \mathbb{A}^n$
such that the orbit $G\cdot x$ of $x$ is not contained in $W$.
  \end{itemize}
  
Then there are constants $k$, $c$ depending only the number, dimension and
degree of the irreducible components of $W$ such that
there are at least $\max(1,c|A|)$ elements
$g\in A^k$ for which $g x\notin W(K)$.
\end{prop}
\begin{proof}[Proof for a special case]
Let us first do the special case of $W$ an irreducible linear subvariety.
We will proceed by induction on the dimension of 
$W$. If $\dim(W)=0$, then $W$ consists of a single point, and the statement is
clear: since
$G\cdot x \not\subset \{x\}$ and  $A$ generates $G$,
there exists a $g\in A$ such that 
$g x \ne x$; if there are fewer than $|A|/2$ such 
elements of $A$, we let $g_0$ be one of them, and note that any
product $g^{-1} g_0$ with $g x = x$ satisfies
$g^{-1} g_0 x \ne x$; there are $> |A|/2$ such products.

Assume, then, that $\dim(W)>0$, and that the statement has been proven for all
$W'$ with $\dim(W')<\dim(W)$. If $g W = W$ for all $g\in A$, then either
(a) $g x$ does not lie on $W$ for any $g\in A$, proving
the statement, or
(b) $g x$ lies on $W$ for every $g\in G = \langle A\rangle$, contradicting
the assumption. Assume that $g W \ne W$ for some $g\in A$; then $W'=gW\cap W$
is an irreducible linear variety with $\dim(W') < \dim(W)$. Thus,
by the inductive hypothesis, there are at least $\max(1, c' |A|)$ elements
$g'\in A^{k'}$ ($c'$, $k'$ depending only on $\dim(W')$) such that
$g' x$ does not lie on $W' = g W\cap W$. Hence, for each such $g'$, either 
$g^{-1} g' x$ or $g' x$ does not lie on $W$. We have thus proven the
statement with $c = c'/2$, $k = k'+1$.
\end{proof}

\begin{prob}
  Generalize the proof so that it works without the assumptions that $W$ be linear or irreducible. Sketch: work first towards removing the
  assumption of irreducibility. Let $W$ be the union
of $r$ components, not necessarily all of the same dimension. The intersection 
$W' = g W \cap W$ may also have several components, but no more than $r^2$;
this is what we meant by ``keeping the degree under control''.
Now pay attention to $d$, the maximum of the dimensions of the components
of a variety, and $m$, the number of components of maximal dimension.
Show that either (1) $d$ is lower for $W'= g W\cap W$ than for $W$, or
(2) $d$ is the same in both cases, but $m$ is lower for $W'$ than for $W$,
or (3)  $x$ does not lie in any component of $W$ of dimension $d$,
and thus we may work instead with $W$ with those components removed.
Use this fact to carry out the inductive process.

Now note that you never really used the fact that $W$ is linear. Instead
of keeping track of the number of components $r$, keep track of the sum of
their degrees. Control that using the generalized form
(\ref{eq:bezout}) of B\'ezout's theorem.
\end{prob}
        %  Nuestra pregunta es entonces la siguiente: si consideramos un
  %  conjunto $A\subset \SL_2(K)$ y una variedad $V$ en $M_2$, como evoluciona
  %  la intersecci\'on $A^k \cap V(K)$ cuando $k$ crece? O, alternativamente:
  %  si $A$ es tal que $|A^k|$ crece lentamente (digamos, $|A^3|<|A|^{1+\delta}$),
 %   que podemos concluir acerca de $|A^k \cap V(K)|$?

\subsection{Dimensional estimates}
By a {\em dimensional estimate} we mean a lower or upper bound on an 
intersection of the form $A^k \cap V$, where $A\subset G(K)$,
$V$ is a subvariety of $G$ and $G/K$ is an algebraic group. As you
will notice, the bounds that we obtain will be meaningful when $A$ grows 
relatively slowly. However, no assumption on $A$ is made, other than that it
generate $G(K)$.

Of course, Proposition \ref{prop:huru} may already be seen as a dimensional
estimate of sorts, in that it tells us that $\gg |A|$ elements of
$A^k$, $k$ bounded, lie outside $W$. We are now
aiming at much stronger bounds; Proposition \ref{prop:huru} will be a useful
tool along the way.

We aim for the estimates whose most general form is as follows.
\begin{theorem}\label{thm:lp}
  Let $G<\GL_n$ be a simple linear algebraic group over a finite field $K$. Let
  $A\subset G(K)$ be a finite set of generators of $G(K)$. Assume
  $A=A^{-1}$, $e\in A$. Let $V$ be a pure-dimensional subvariety of $G$.
  Then
  \begin{equation}\label{eq:utur}
    |A\cap V(K)|\ll |A^k|^{\frac{\dim V}{\dim G}},\end{equation}
  where $k$ and the implied constant depend only on $n$ and on $\deg(V)$.
\end{theorem}

Estimates of this form can be traced in part to
\cite{LP} ($A$ a subgroup, $V$ general) and in part to \cite{Hel08}
y \cite{HeSL3} ($A$ an arbitrary set, but $V$ special).
We now have Theorem \ref{thm:lp}, thanks to 
\cite{BGT} and \cite{MR3402696}.
In fact, \cite{MR3402696} gives a
more general statement, in that twisted
groups of Lie type are covered. Actually, one can state Theorem \ref{thm:lp}
in an even more general form, in that the assumption that $K$ is finite can
be dropped, and the condition that $A$ generate $G(K)$ can be replaced
by a condition that $\langle A\rangle$ be ``Zariski-dense enough'', meaning
not contained in a union of $\leq C$ varieties of degree $\leq C$, where
$C$ depends only on $n$ and $\deg(V)$.
%, with a dependence on their rank rather than on
%$n$. For Chevalley groups, dependence on the rank and on $n$ are the same.)
%($A$ an arbitrary set, $V$ an arbitrary
%subvariety of $G$, and $G$ a simple linear algebraic group, as in
%\cite{LP}). Here $k$ is a constant that may depend on the number and
%degree of the components of $V$, and on the rank of $G$ (e.g.,
%$G=\SL_n$ has rank $n-1$), but not on the field $K$ or on the set $A$.

We will show how to prove the estimate (\ref{eq:utur}) in the
case we actually need, but in a way that can be generalized to arbitrary
$V$ and arbitrary simple $G$. We will give a detailed outline of
how to obtain the generalization.

Actually, as a first step towards the general strategy, let us study a
particular $V$ that we will not use in the end; it was crucial in earlier
versions of the proof, and, more importantly, it makes several of the
key ideas clear quickly. The proof is basically the same as in 
\cite[\S 4]{Hel08}. In particular, it will not look as if we used any
algebraic geometry; however, the concrete procedure
we follow here will then lead us naturally to a general procedure that will ask for the language and the basic tools of algebraic geometry.

 \begin{lem}\label{lem:rokto}
   Let $G = \SL_2$, $K$ a field. Let $A\subset G(K)$ be a finite set of
   generators of $G(K)$. Assume $A=A^{-1}$, $e\in A$. Let $T$
   be a maximal torus of $G$. Then
\begin{equation}\label{eq:hobo}|A \cap T(K)| \ll |A^k|^{1/3},\end{equation}
where $k$ and the implied constant are absolute.
 \end{lem}
 \begin{proof}
   We can assume without loss of generality that
   $|K|$ and $|A|$ are greater than a constant,
   as otherwise the statement is trivial.
%   We can also assume without loss of generality that
%   $A = A^{-1}$, $e\in A$, and $|A|$ is greater than a constant, simply
%   by replacing $A$ by
%   $(A \cup A^{-1}\cup \{e\})^c$, $c$ a constant, if necessary.
   We can also write the elements of $T$ as diagonal matrices, by conjugation
   by an element of $\SL_2(\overline{K})$.

  Let \begin{equation}\label{eq:murut}
    g=\left(\begin{matrix} a & b\\ c & d\end{matrix}\right)\end{equation}
be any element of $\SL_2(\overline{K})$ with
    $a b c d\ne 0$. Consider the map
  $\phi:T(K) \times T(K) \times T(K) \to G(K)$ given by
\[\phi(x, y, z) = x \cdot g y g^{-1} \cdot z.\]
We would like to show that this map is in some sense almost injective.
(What for? If the map were injective, and we had
  $g\in A^\ell$, $\ell$ bounded by a constant, we would have
 \[|A\cap T(K)|^3 = |\phi(A\cap T(K),A\cap T(K),A\cap T(K))| 
 \leq |A A^\ell A A^{-\ell} A| = |A^{2\ell+3}|,\]
 which would imply immediately the result we are trying to prove.
 Here we are simply using the fact that the image $\phi(D)$ of an injection
 $\phi$ has the same number of elements as the domain $D$.)
 
Multiplying matrices, we see that, for
\[x=\left(\begin{matrix} r & 0 \\0 & r^{-1}\end{matrix}\right),\;\;
y = 
\left(\begin{matrix} s & 0 \\0 & s^{-1}\end{matrix}\right),\;\;
z = 
  \left(\begin{matrix} t & 0 \\0 & t^{-1}\end{matrix}\right),\]
$\phi((x,y,z))$  equals
\begin{equation}\label{eq:malvot}  \left(\begin{matrix} r t (s a d - s^{-1} b c) &
    r t^{-1} (s^{-1} - s) a b\\
    r^{-1} t (s - s^{-1}) c d &
r^{-1} t^{-1} (s^{-1} a d - s b c)\end{matrix}\right).\end{equation}
Let $s\in \overline{K}$ be such that $s^{-1} - s \ne 0$ and
$s a d - s^{-1} b c \ne 0$.
A brief calculation shows then that
$\phi^{-1}(\{\phi((x,y,z))\})$ has at most $16$ elements: we have
\[  r t^{-1} (s^{-1} - s) a b \cdot 
r^{-1} t (s - s^{-1}) c d  = - (s - s^{-1})^2 a b c d,\]
and, since $a b c d \ne 0$, at most $4$ values of $s$
can give the same value $- (s-s^{-1})^2 a b c d$
(the product of the top right and bottom left entries of
((\ref{eq:malvot})); for each such value of $s$,
the product and the quotient of the upper left and upper right entries of
(\ref{eq:malvot}) determine $r^2$ and $t^2$, respectively,
and obviously there are at most $2$ values of $r$ and $2$ values of $t$
for $r^2$, $t^2$ given.

Now, there are at most $4$ values of $s$ such that
    $s^{-1} - s = 0$ or $s a d - s^{-1} b c = 0$. Hence,
    \[|\phi(A\cap T(K),A\cap T(K),A\cap T(K))| \geq
    \frac{1}{16} |A\cap T(K)| (|A\cap T(K)| - 4) |A\cap T(K)|,\]
    and, at the same time, $\phi(A\cap T(K),A\cap T(K),A\cap T(K))\subset
    A A^\ell A A^{-\ell} A = A^{3+2\ell}$, as we said before. If
    $|A\cap T(K)|$ is less than $8$ (or any other constant),
    conclusion (\ref{eq:hobo}) is trivial. Therefore,
        \[|A\cap T(K)|^3 \leq 2 |A\cap T(K)| (|A\cap T(K)| - 4) |A\cap T(K)|
    \leq 32 |A^{2\ell+3}|,\]
    i.e., (\ref{eq:hobo}) holds.

    It only remains to verify that there exists an element
    (\ref{eq:murut}) of $A^\ell$ with $a b c d\ne 0$. Now,
    $a b c d = 0$ defines a subvariety $W$ of
    $\mathbb{A}^4$, where $\mathbb{A}^4$ is identified with the space
    of $2$-by-$2$ matrices. Moreover, for $|K|>2$, there are elements
    of $G(K)$ outside that variety. Hence, the conditions of
    Prop.~\ref{prop:huru} hold (with $x=e$).
    Thus, we obtain that there is a
    $g\in A^\ell$ ($\ell$ a constant)
    such that $g\not\in W(K)$, and that was what we needed.
    \end{proof}

 Let us abstract the essence of what we have just done, so that we can
 then generalize the result to an arbitrary variety $V$ instead of working
 just with $T$. For the sake of convenience, we will do the case
 $\dim V = 1$, which is, at any rate, the case we will need.
 The strategy of the proof of Lemma \ref{lem:rokto} is to construct a
 morphism $\phi:V\times V\times \dotsb \times V\to G$
($r$ copies of $V$, where $r = \dim(G)$) of the form
\begin{equation}\label{eq:kukuku}
  \phi(v_1,\dotsc,v_{r}) = v_1 g_1 v_2 g_2 \dotsb v_{r-1} g_{r-1} v_r,\end{equation}
where $g_1,g_2,\dotsc,g_{r-1}\in A^\ell$, in such a way that, for
$v = (v_1,\dotsc,v_{r})$ a generic
point in $V\times V\times \dotsb \times V$, the preimage
$\phi^{-1}(\{\phi(v)\})$ has dimension $0$. Actually, as we have just seen,
it is enough to prove that this is true for $(g_1,g_2,\dotsc,g_{r-1})$
a generic element of $G^{r-1}$; the escape argument (Prop.~\ref{prop:huru}) takes care of the rest.

The following lemma is the same as \cite[Prop. 5.5.3]{MR3309986},
which, in turn, is the same as \cite[Lemma 4.5]{LP}.
We will give a proof valid for $\mathfrak{g}$ simple.

\begin{lem}\label{lem:tush}
  Let $G<\SL_n$ be a simple algebraic group defined over
  a field $K$.
  Let $V, V'\subsetneq G$ be irreducible subvarieties with $\dim(V)<\dim(G)$
  and $\dim(V')>0$.
  Then, for every
  $g\in G(\overline{K})$ outside a subvariety $W\subsetneq G$ depending
  on $V$ and $V'$,
the variety $\overline{V g V'}$ has dimension $>\dim(V)$.

  Moreover, the number and degrees of the irreducible components of $W$
  are bounded by a constant that depends only on $n$ and $\deg(V)$ and
  $\deg(V')$.
\end{lem}
In fact, the proof we will now see bounds the number and degrees
of the components of $W$ in terms of
$n$ alone.
\begin{proof}[Proof for $\mathfrak{g}$ simple]
  We can assume without loss of generality -- replacing $V$ and $V'$ by
  varieties $V h$ and $h' V$, $h,h'\in G(\overline{K})$, if necessary -- that $V$ and $V'$ go through
  the origin, and that the origin is a non-singular point for $V$ and $V'$.
We may also assume without loss of generality that $K$ is algebraically closed.
  
  Let $\mathfrak{v}$ and $\mathfrak{v}'$ be the tangent spaces to $V$ and $V'$
  at the origin. The tangent space to
  $\overline{V g V'} g^{-1}$ at
  the identity is
  $\mathfrak{v} + \Ad_g \mathfrak{v}'$. Thus, for $\overline{V' g V}$
  to have dimension
  $> \dim(V)$, it is enough that $\mathfrak{v} + \Ad_g \mathfrak{v}'$ have
  dimension $> \dim(\mathfrak{v}) = \dim(V)$.

  Suppose that this is not the case for any $g$ on $G$. Then the space
  $\mathfrak{w}$ spanned
  by all spaces $\Ad_g \mathfrak{v}'$, for all $g$, is contained in
  $\mathfrak{v}$. Since $\dim(V)<\dim(G)$,
  $\mathfrak{v} \subsetneq \mathfrak{g}$.
  Clearly, $\mathfrak{w}$ is non-empty and invariant
  under $\Ad_g$ for every $g$. Hence it is an ideal. However, we are assuming
  $\mathfrak{g}$ to be simple. Contradiction.

  Thus, $\mathfrak{v} + \Ad_g \mathfrak{v}'$ has
  dimension greater than $\dim(\mathfrak{v})$ for some $g$. It is easy to
  see that the points $g$
  where that is not the case are precisely those such that all
  $(\dim(\mathfrak{v})+1)\times (\dim(\mathfrak{v})+1)$ minors
  of a matrix -- whose entries are polynomial on the entries of $g$ -- vanish.
  We let $W$ be the subvariety of $V$ where those minors all vanish.
  The claim on the number and degrees of components of $W$ follows by
  B\'ezout (\ref{eq:bezout}).
  %Hence, the points where
  %$\dim(\mathfrak{v} + \Ad_g \mathfrak{v}')>\dim(\mathfrak{v})$ are the
  %points outside a variety $W\subsetneq V$ given by a system of
  %equations whose degree and number is bounded in terms of $n$ alone.
\end{proof}

We can now generalize our proof of Lemma \ref{lem:rokto}, and thus prove
(\ref{eq:utur}) for all varieties of dimension $1$. Before we start,
we need a basic counting lemma, left as an exercise.

\begin{prob}\label{prob:hejhoj}
  Let $W\subset \mathbb{A}^n$ be a variety defined over $K$ such that
  every component of $W$ has dimension $\leq d$. Let $S$ be a finite
  subset of $K$.
  Then the number of points $(x_1,\dotsc,x_n) \in S\times S \times \dotsb
  \times S$ ($n$ times) lying
  on $W$ is $\ll |S|^d$, where the implied constant depends only on $n$
  and on the number and degrees of the components of $W$.
%  {\em Sketch/hints:}
%    we can assume without loss of generality that $V$ is irreducible.
%    We will reduce the problem either to a case with lower $d$, or to
%    a case with the same $d$ and lower $n$; we will work with projections,
%    so degrees will keep themselves bounded. Let
%     $\pi:V\to \mathbb{A}^{n-1}$
%    be the projection to the first $n-1$ coordinates. The
%    derivative $D\pi$ is singular on a subvariety $W$ of $V$.
%    If $W=V$, then $\dim(\pi(V))<\dim(V)$; we have reduced the problem to
%    one of lower $d$ and $n$, and finish matters by the trivial fact that
%    at most $|S|$ points in $S\times \dotsb \times S$ ($n$ times)
%    can be projected to a single point
%    by $\pi$. If $W\subsetneq V$, then, since $V$ is irreducible,
%    $\dim(W)<\dim(V)$. Bound separately the number of points on $W(K)$
%    (lower $d$) and the number of points on $V(K)\setminus W(K)$;
%    the latter is at most the degree of $\pi$ (bounded by the degree of $W$)
%    times the number of points on $\overline{\pi(V(K))}$ (lower $n$).
\end{prob}

\begin{prop}\label{prop:cheyen}
  Let 
  $G\subset \SL_n$ be an simple algebraic group over a finite field
  $K$. Assume that
  $|G(K)|\geq c |K|^{\dim(G)}$, $c>0$.
  Let $Z\subset G$ be a variety of dimension $1$.
  Let $A\subset G(K)$ be a set of generators of $G(K)$ such that
  $A=A^{-1}$, $e\in A$.

  Then
    \begin{equation}\label{eq:clada}
      |A\cap Z(K)|\ll |A^k|^{1/\dim(G)},\end{equation}
    where $k$ and the implied constant depend only on $n$, $c$,
    $\deg(G)$ and the
    number and degrees of the irreducible components of $Z$.
    \end{prop}
Obviously, $G = \SL_n$ is a valid choice, since it is simple and
$|\SL_n(K)|\gg |K|^{n^2-1} = |K|^{\dim(G)}$.
\begin{proof}
  We will use Lemma \ref{lem:tush} repeatedly. When we apply it,
  we get a subvariety $W\subsetneq G$ such that, for every $g$
  outside $W$, some component of $\overline{V g V'}$ has dimension
  $>\dim(V)$ (where $V$ and $V'$ are varieties satisfying the conditions
  of Lemma \ref{lem:tush}). Since $G$ is irreducible,
  every component of $W$ has dimension less than $\dim(G)$.
  By Exercise \ref{prob:hejhoj} (with $S=K$)
  and the assumption $|G(K)|\geq c |K|^{\dim(G)}$, 
  there is at least one point of $G(K)$ not on $W$, provided that $|K|$
  is larger than a constant, as we can indeed assume. Hence, we can
  use escape from subvarieties (Prop.~\ref{prop:huru}) to show that
  there is a $g\in (A\cup A^{-1} \cup \{e\})^\ell$, where $\ell$ depends
  only on the number and degrees of components of $W$, that is to say -- by
  Lemma \ref{lem:tush} -- only on $n$ and $\deg(G)$.

  So: first, we apply Lemma \ref{lem:tush} with $V=V'=Z$; we obtain
  a variety $V_2 = \overline{V g_1 V'} = \overline{Z g_1 Z}$ with
  $g\in (A\cup A^{-1} \cup \{e\})^\ell$ such that $V_2$ has at least
  one component of dimension $2$. (We might as well assume $V$ is
  irreducible from now on; then $V_2$ is irreducible.)
  We apply Lemma \ref{lem:tush} again with $V=V_2$, $V'=Z$, and obtain
  a variety $V_3 = \overline{V_2 g_2 Z} = \overline{Z g_1 Z g_2 Z}$
  of dimension $3$. We go on and on, and get that there are
  $g_1,\dotsc,g_{m-1}\in (A\cup A^{-1} \cup \{e\})^{\ell'}$, $r=\dim(G)$, such that
  $\overline{Z g_1 Z g_2 \dotsc Z g_{r-1} Z}$ has dimension $r$.

  Hence, the variety $W$ of singular points of the
  map $f$ from $Z^r=Z\times Z\times \dotsb \times Z$ ($r$ times) to
  $G$ given by
  \[f(z_1,\dotsc,z_m) = z_1 g_1 z_2 g_2 \dotsc z_{r-1} g_{r-1} z_r\]
  cannot be all of $Z\times \dotsc \times Z$.
  Thus, since $Z\times \dotsc \times Z$ is irreducible, every component of
  $W$ is of dimension less than $\dim V$.
    Again by Exercise \ref{prob:hejhoj} (with $S = A\cap Z(K)$),
    at most $O(|A\cap Z(K)|^{r-1})$ points
    of $(A\cap Z(K))\times \dotsb \times (A\cap Z(K))$ ($r$ times)
    on $W$. The number of points of $(A\cap Z(K))\times \dotsb \times
    (A\cap Z(K))$ not on $W$ is at most the degree of $f$ times the number of points
    on $f(A\cap Z(K),\dotsc,A\cap Z(K))$, which is contained in
    $A^k$ for $k = {r+(r-1) \ell'}$. Therefore,
    \[|A\cap Z(K)|^r \leq
    \deg(f) |A^k| + O\left(|A\cap Z(K)|^{r-1}\right),\]
    and so we are done.
\end{proof}

In general, one can prove (\ref{eq:utur}) for $\dim(V)$ arbitrary
using very similar arguments, together with an induction on the dimension
of the variety $V$ in (\ref{eq:utur}). We will demonstrate the basic
procedure doing things in detail for $G=\SL_2$ and for the kind of variety
$V$ for which we really need to prove estimates.

We mean the variety 
$V_t$ defined by
\begin{equation}\label{eq:schlingue}
  \det(g) = 1, \tr(g) = t
    \end{equation}
for $t\ne \pm 2$. Such varieties are of interest to us because, for any
regular semisimple $g\in \SL_2(K)$ (meaning: any matrix in
$\SL_2(K)$ having two distinct eigenvalues), the conjugacy class
  $\Cl(g)$ is contained in $V_{\tr(g)}$.

%  Se pueden encontrar varias pruebas en la literatura. La siguiente trata
%  de ser relativamente conceptual, sin al mismo tiempo realmente
%  tratar de ser tal
%  que se generalize con toda facilidad\footnote{Para una prueba con constantes
%  expl\'icitas, v\'ease \cite[Thm.~3.11]{MR3144176}. La prueba en
%  \cite[\S 5.2]{MR3348442} puede generalizarse m\'as f\'acilmente.
%  Ver tambi\'en la prueba en [Prop.~1.5.17]\cite{MR3309986}.}.

  \begin{prop}\label{prop:juru}
    Let $K$ be a finite field.
    Let $A\subset \SL_2(K)$ be a set of generators of
    $\SL_2(K)$ with $A=A^{-1}$, $e\in A$.
    Let $V_t$ be given by (\ref{eq:schlingue}).
    
    Then, for every $t\in K$ other than $\pm 2$,
    \begin{equation}\label{eq:terka}
      |A\cap V_t(K)|\ll |A^k|^{\frac{2}{3}},
    \end{equation}
    where $k$ and the implied constant are absolute.
  \end{prop}
  Needless to say,
  $\dim(\SL_2) = 3$ and $\dim(V_t) = 2$, so this is a special case of
  (\ref{eq:utur}).

  \begin{proof}
    Consider the map $\phi:V_t(K)\times V_t(K) \to
    \SL_2(K) $ given by
    \[\phi(y_1,y_2) = y_1 y_2^{-1}.\]
    It is clear that \[\phi(A\cap V_t(K),A\cap V_t(K))\subset A^2.\]
    Thus, if $\phi$ were injective, we would obtain immediately that
    $|A\cap V_t(K)|^2\leq |A^2|$. Now, $\phi$ is not injective, not even nearly
    so. The preimage of $\{h\}$, $h\in \SL_2(K)$, is
    \[\phi^{-1}(\{h\}) = \{(w,h^{-1} w): \tr(w) = t, \tr(h^{-1} w) = t\}.\]

    We should thus ask ourselves how many elements of $A$ lie on the
    subvariety
        $Z_{t,h}$ of $G$ defined by
    \[Z_{t,h} =  \{(w,h w): \tr(w) = t, \tr(h^{-1} w) = t\}.\]
    For $h\ne \pm e$, $\dim(Z_{t,h}) = 1$, and the number
    and degrees of irreducible components of $Z_{t,h}$
    are bounded by an absolute constant. Thus, applying Proposition
        \ref{prop:cheyen}, we get that, for $h\ne \pm e$,
    \[|A\cap Z_{t,h}(K)|\ll |A^{k'}|^{1/3},\]
    where $k'$ and the implied constant are absolute.

    Now, for every $y_1\in V_t(K)$, there are at least $|V_t(K)|-2$
    elements $y_2\in V_t(K)$ such that $y_1 y_2^{-1} \ne \pm e$.
    We conclude that
    \[|A\cap V(K)| (|A\cap V(K)|-2) \leq |A^2|\cdot \max_{g\ne \pm e}
    |A\cap Z_{t,h}(K)| \ll |A^2| |A^{k'}|^{1/3}.\]
    We can assume that $|A\cap V(K)|\geq 3$, as otherwise the desired
    conclusion is trivial. We obtain, then, that
    \[|A\cap V(K)|\ll |A^k|^{2/3}\]
    for $k = \max(2,k')$, as we wanted.
      \end{proof}
  
%  En verdad, la \'unica parte de la prueba que fue lig\'eramente {\em ad hoc}
%  fue el lugar donde constru\'imos $g_1$, $g_2$ para los cuales
%  $W_t\cap g_1 W_t \cap g_1 g_2 W_t$ es $0$-dimensional. Tendr\'iamos
%  que hacerlo caso por caso para cada familia de grupos lineares
%  ($\SL_n$, $\SO_n$, $\Sp_{2n}$, etc.). La alternativa es mostrar que la
%  existencia de tales $g_1,g_2,\dotsc$ es una consecuencia de la simplicidad.

  Now we can finally prove the result we needed.
  \begin{corollary}\label{cor:hutz}
    Let $G=\SL_2$, $K$ a finite field. Let $A$ be a set of generators of
    $G(K)$ with $A=A^{-1}$, $e\in A$.
    Let $g\in A^\ell$ ($\ell\geq 1$) be regular semisimple. Then
    \begin{equation}\label{eq:hop1}
      |A^2\cap C(g)|\gg \frac{A}{|A^{k \ell}|^{2/3}},\end{equation}
    where $k$ and the implied constant are absolute.
    
    In particular, if $|A^3|\leq |A|^{1+\delta}$, then
    \begin{equation}\label{eq:hop2}
      |A^2 \cap C(g)|\gg_\ell |A|^{1/3-O(\ell \delta)}.\end{equation}
  \end{corollary}
  \begin{proof}
    Proposition \ref{prop:juru} and Lemma \ref{lem:lawve}
    imply (\ref{eq:hop1}) immediately, and
    (\ref{eq:hop2}) follows readily from
    (\ref{eq:hop1}) via (\ref{eq:marmundo}).
    \end{proof}

  Let us now see two problems whose statements we will not use; they
  are, however, essential if one wishes to work in $\SL_n$ for $n$
  arbitrary, or in an arbitrary simple algebraic group.
  The first problem is challenging, but we have already seen and applied
  the main ideas involved in its solution. In essence, it is a matter
  of setting up a recursion properly.
  
  \begin{prob}\label{exer:corodo}  
    Generalize Proposition \ref{prop:cheyen} to pure-dimensional
    varieties $Z$ of arbitrary dimension; that is, prove Theorem \ref{thm:lp}.
  \end{prob}

%  As we said before in passing, an element $g\in \SL_n(K)$ is
%  {\em regular semisimple} if it has $n$ distinct eigenvalues.
%  For $G=\SL_n$ and $g\in G(K)$, the elements of $C(g)$ are the points
%  $T(K)$ of an abelian algebraic subgroup $T$ of $G$, called a
%  maximal torus. Just as in the case of $\SL_2$, it is always possible
%  to conjugate by some element of $G(\overline{K})$ so that $T$ becomes
%  simply the group of diagonal matrices. It is thus clear that
%  $\dim(T) = n-1$; it is also easy to see that
%  $\dim(\overline{\Cl(g)})=\dim(G) - \dim(T)$, since $\overline{\Cl(g)}$
%  is just the variety consisting of the matrices in $G$ having the same
%  eigenvalues as $g$.

  The following exercise is easy. In part (\ref{it:parobo}), follow the proof of Corollary \ref{cor:hutz}, using Exercise \ref{exer:corodo}.
  \begin{prob}\label{exer:lowbou}
    Let $G$ be a simple algebraic group over a finite field $K$.
    Let $A\subset G(K)$, $A=A^{-1}$, $e\in A$, $\langle A\rangle = G(K)$.
    Let $g\in A^\ell$, $\ell\geq 1$.
    \begin{enumerate}
    \item  Using the material in \S
  \ref{subsub:morph}, show that $\dim G - \dim \overline{\Cl(g)} = \dim C(g)$.
    \item\label{it:parobo}
      Show that, if $|A|^3\leq |A|^{1+\delta}$,
        \begin{equation}\label{eq:hosop2}
      |A^2 \cap C(g)|\gg |A|^{\frac{\dim(C(g))}{\dim(G)}-O(\ell \delta)},\end{equation}
        where the implied constants depend only on $n$.
        \end{enumerate}
  \end{prob}
  If $g$ is regular semisimple, then, as we know, $C(g)$ is a maximal torus.
%  Since all elements outside a proper subvariety of $G$ are regular semisimple,
%  escape from subvarieties guarantees that there is at least one
  \section{Growth and diameter in $\SL_2(K)$}\label{sec:growthth}
\subsection{Growth in $\SL_2(K)$, $K$ arbitrary}
We come to the proof of our main result. Here we will be closer
to newer treatments (in particular, \cite{MR3402696})
than to what was the first proof, given in \cite{Hel08};
these newer versions generalize more easily. We will
give the proof only for $\SL_2$, and point out the couple of places in the proof
where one would has to be especially careful when generalizing matters to
$\SL_n$, $n>2$, or other linear algebraic groups.

The proof in \cite{Hel08} used the sum-product theorem (Thm.~\ref{thm:orb}).
We will not use it, but the idea of ``pivoting'' will reappear. It is
also good to note that, just as before, there is an inductive process here,
carried out on a group $G$, even though $G$ does not have a natural order
($1,2,3,\dotsc$).
All we need for the induction to work is a set of generators
$A$ of $G$.

 \begin{theorem}[Helfgott \cite{Hel08}]\label{thm:main08}
   Let $K$ be a finite field. Let $A\subset \SL_2(K)$ be a set of
   generators of
   $\SL_2(K)$ with $A=A^{-1}$, $e\in A$. There either
   \begin{equation}\label{eq:filisteo}|A^3|\geq |A|^{1+\delta},\end{equation}
   where $\delta>0$ is an absolute constant, or
   \begin{equation}\label{eq:asmoneo}A^3 = \SL_2(K).\end{equation}
 \end{theorem}
 Actually, \cite{Hel08} proved this result (with $A^k$, $k$ a constant,
 instead of $A^3$ in (\ref{eq:asmoneo})) for $K = \mathbb{F}_p$; the first
 generalization to a general finite field $K$ was given by \cite{MR2788087}.
 The proof we are about to see works for $K$ general without
 any extra effort. It works, incidentally, for $K$ infinite as well,
 dropping the condition $|A|<|\SL_2(K)|^{1-\epsilon}$, which becomes trivially
 true.
 The case of characteristic $0$ is actually easier than the case
 $K = \mathbb{F}_p$;
 the proof in \cite{Hel08} was already valid for $K = \mathbb{R}$ or
 $K = \mathbb{C}$, say. However, for applications,
 the ``right'' result for $K=\mathbb{R}$
 or $K=\mathbb{C}$ is not really Thm.~\ref{thm:main08}, but a statement
 counting how many elements there can be in $A$ and $A\cdot A\cdot A$ that
 are separated by a given small distance from each other; that was proven
 in \cite{BGSU2}, adapting the techniques in \cite{Hel08}. 
 \begin{proof}
   We may assume that $|A|$ is larger than an absolute constant, since
   otherwise the conclusion would be trivial. Let $G = \SL_2$.
   
   Suppose that $|A^3|<|A|^{1+\delta}$, where $\delta>0$ is a small constant
   to be determined later. By escape (Prop.~\ref{prop:huru}),
   there is an element $g_0\in A^c$ that is regular semisimple
   (that is, $\tr(g_0)\ne \pm 2$), where $c$ is an absolute constant.
   (Easy exercise: show we can take $c=2$.)
   Its centralizer in $G(K)$
   is $\mathbf{T} := C(g) = T(\overline{K})\cap G(K)$ for some maximal torus $T$.

   Call $\xi\in G(K)$ a {\em pivot}
   if the map $\phi_\xi:A\times \mathbf{T}
   \to G(K)$ defined by
   \begin{equation}\label{eq:naksym}
     (a,t) \mapsto a \xi t \xi^{-1}
   \end{equation} is injective as a function from $(\pm e\cdot
   A)/\{\pm e\} \times \mathbf{T}/\{\pm e\}$
   to $G(K)/\{\pm e\}$.

   {\em Case (a): There is a pivot $\xi$ in $A$.} By Corollary
   \ref{cor:hutz}, there are $\gg |A|^{1/3 - O (c\delta)}$ elements of $\mathbf{T}$
   in $A^{-1} A$. Hence, by the injectivity of $\phi_\xi$,
   \[\left|\phi_\xi(A,A^2\cap \mathbf{T})\right|
   \geq \frac{1}{4} |A| |A^2\cap \mathbf{T}|
   \gg |A|^{\frac{4}{3} - O(c \delta)}.\]
   At the same time, $\phi_\xi(A,A^2\cap \mathbf{T})\subset A^5$, and thus
   \[|A^5|\gg |A|^{4/3 - O(c \delta)}.\]
   For $|A|$ larger than a constant and
   $\delta>0$ less than a constant,
   this inequality gives us a contradiction with
    $|A^3| < |A|^{1+\delta}$ (by Ruzsa
   (\ref{eq:jotor})).

   {\em Case (b): There are no pivots $\xi$ in $G(K)$.} Then, for every
   $\xi\in G(K)$, there are $a_1,a_2\in A$, $t_1,t_2\in \mathbf{T}$,
   $(a_1,t_1) \ne (\pm a_2, \pm t_2)$ such that $a_1 \xi t_1 \xi^{-1} =
   \pm e \cdot a_2 \xi t_2 \xi^{-1}$, and that gives us that
   \[a_2^{-1} a_1 = \pm e\cdot \xi t_2 t_1^{-1} \xi^{-1}.\]
   In other words, for each
   $\xi\in G(K)$, $A^2$ has a non-trivial intersection with the torus
   $\xi T \xi^{-1}$:
   \begin{equation}\label{eq:nortsch}
     A^2 \cap \xi \mathbf{T} \xi^{-1} \not\subset \{\pm e\}.
   \end{equation}
   (Note this means that case (b) never arises for $K$ infinite. Why?)

   Choose any $g\in A^2 \cap \xi \mathbf{T} \xi^{-1}$ with $g\ne \pm e$. Then
   $g$ is regular semisimple. (This fact is peculiar to $\SL_2$, or rather
   to groups of rank $1$. This is
   one place in the proof that requires some work when you generalize it to
   other groups.)

   The centralizer 
   $C(g)$ of $g$ equals $\xi \mathbf{T} \xi^{-1}$ (why?). Hence, by
   Corollary \ref{cor:hutz}, we obtain that there are
   $\geq c' |A|^{1/3 - O(\delta)}$ elements of $\xi \mathbf{T} \xi^{-1}$
   in $A^2$,
   where $c'$ and the implied constant are absolute.

   At least $(1/2) |G(K)|/|\mathbf{T}|$ maximal tori of $G$ are of the form
   $\xi T \xi^{-1}$, $\xi \in G(K)$ (check this yourself!).
   Every semisimple element of $G$ that is not $\pm e$ is regular
   (again, something peculiar to $\SL_2$); thus, every element of $G$
   that is not $\pm e$ can lie on at most one maximal torus. Hence
   \[|A^2|\geq \frac{1}{2} \frac{|G(K)|}{|\mathbf{T}|} (c |A|^{1/3 - O(\delta)} - 2)
   \gg |G(K)|^{2/3} |A|^{1/3 - O(\delta)}.\]

   Therefore, either $|A^2|> |A|^{1+2\delta}$ (say) or
   $|A|\geq |G|^{1-O(\delta)}$. In the first case,
      we have obtained a contradiction. In the second case, Proposition
\ref{prop:diplo} implies that $A^3 = G$.   

{\em Case (c): There are pivots and non-pivots in $G(K)$.}
Since $\langle A\rangle = G(K)$, this implies that there exists a
non-pivot $\xi\in G$ and an
$a\in A$ such that $a\xi \in G$ is a pivot.
Since $\xi$ is not a pivot, (\ref{eq:nortsch}) holds, and thus there are
$|A|^{1/3-O(\delta)}$ elements of $\xi \mathbf{T} \xi^{-1}$ in $A^k$.

  At the same time, $a\xi$ is a pivot, i.e., the map
  $\phi_{a\xi}$ defined in (\ref{eq:naksym}) is injective
  (considered as an application from
$A/\{\pm e\} \times \mathbf{T}/\{\pm e\}$
   to  $G(K)/\{\pm e\}$).
  Therefore,
  \[\left|\phi_{a \xi}(A, \xi^{-1} (A^2 \cap \xi T \xi^{-1}) \xi)\right| \geq \frac{1}{4}
|A| |A^2 \cap \xi T \xi^{-1}| \geq \frac{1}{4} |A|^{\frac{4}{3} - O(\delta)}.\]
Since $\phi_{a \xi}(A, \xi^{-1} (A^2 \cap \xi T \xi^{-1}) \xi) \subset A^5$,
we obtain that \begin{equation}\label{eq:matameri}
|A^5|\geq \frac{1}{4} |A|^{4/3 - O(\delta)}.\end{equation}
Thanks again to Ruzsa (\ref{eq:jotor}), this inequality contradicts
$|A^3|\leq |A|^{1+\delta}$ for $\delta>0$ smaller than a constant.
 \end{proof}
 The following is a trivial exercise.
 \begin{prob}
   Using Theorem \ref{thm:main08}, show that the statement of
   Thm.~\ref{thm:main08} is also true with $\PSL_2$ in place of $\SL_2$.
   This step finishes the proof of Thm.~\ref{thm:main}.
 \end{prob}

 For $\SL_n$, $n>2$, or for general algebraic groups, there is, as we have
 seen, one difficulty in generalizing the above proof: a semisimple element
 other than $\pm e$ is not necessarily regular. The key to circumventing
 this difficulty is to use Theorem \ref{thm:lp} to bound the number
 of elements on non-maximal subtori of a maximal torus $T$, and, in that way,
 bound the number of non-semisimple elements of $A^k$ on $T$.

 \begin{prob}
   Using this observation, modify the proof of Thm.~\ref{thm:main08}
   so as to work for any simple linear algebraic group $G$.
 \end{prob}
 %-- by considerably
 %less than the lower bound on $|A^2\cap T|$ given by Exercise \ref{exer:lowbou}.
 %Ese teorema a\'un se usaba, si bien de manera m\'as indirecta,
% en \cite{Hel08}. Hoy en d\'ia ya se sabe evitar ese resultado; empero,
%  su prueba ha dejado 

 %5.4, 5.3(?), 3.2 y 5.5 (como serie de ejercicios?)

  %Mention Babai's conjecture here or later.

There remains the question of what the optimal value of $\delta$ in 
Thm.~\ref{thm:main08} could be.
Kowalski \cite{MR3144176}
proves Thm.~\ref{thm:main08} with $\delta=1/3024$
(under the assumption $A=A^{-1}$). Button and Roney-Dougal prove
(under the same assumption) that one cannot do better than
$\delta=(\log_2 7 - 1)/6 \approx 0.3012$ \cite{buttonroneydougal}.

To obtain a good value of $\delta$, it seems best to aim for a statement
with a conclusion of the form
\[|A^3|\geq c |A|^{1+\delta}\]
instead of (\ref{eq:filisteo}). It may be even better to aim for a result
of the form, say,
\[|A A_0^k A A_0^k A A_0^k A|\geq c |A|^{1+\delta},\]
where $A_0$ is an arbitrary set of generators of $\SL_2(K)$. Then, when
using our result to prove a diameter bound (as in exercise \ref{prob:agant}),
we can set $A_0$ to be our initial set of generators $S$, whereas we set
$A$ equal to increasing powers of $S$. The resulting constant $C$ in the
exponent of the bound $\diam \Gamma(G,S)\ll (\log |G|)^C$ should then
improve substantially over the value $C=3323$ given in \cite{MR3144176}.

 Of course, we still need to prove Prop.~\ref{prop:diplo}.
 Let us do so.
  \subsection{The case of large subsets}\label{sec:subgra}
  Let us first see how $A$ grows when
  $A\subset \SL_2(\mathbb{F}_q)$ is large with respect to
  $G=\SL_2(\mathbb{F}_q)$. In fact, it is not terribly hard to show that, if
  $|A| \geq |G|^{1-\delta}$, $\delta>0$ a small constant,
  then
  $(A \cup A^{-1} \cup \{e\})^k = G$, where $k$ is an absolute constant.
  To proceed as in \cite{Hel08}:
  we can use (\ref{eq:vento}) to pass to the solvable group of upper-
  or lower-triangular matrices, then go on as in \S \ref{subs:affi}
  to show that the subgroups $U^{\pm}$ of upper- or lower-triangular
  matrices are contained in $(A\cup A^{-1} \cup \{e\})^{k'}$, $k'$ a constant;
  we are then done by $G=U^- U^+ U^- U^+$.
  
  We will prove a stronger and nicer result: $A^3=G$. The proof is due to
  Nikolov and Pyber \cite{MR2800484}; it is based on a classical idea,
  brought to bear to this particular context by Gowers \cite{MR2410393}.
  It will give us the opportunity to revisit the adjacency operator
  $\mathscr{A}$ and its spectrum.

  Recall that a {\em complex representation} of a group $G$ is just a
  homomorphism $\phi:G\to \GL_d(\mathbb{C})$; by the {\em dimension}
  of the representation we just mean $d$. A representation
  $\phi$ is {\em trivial} if $\phi(g)=e$ for every $g\in G$.

   The following result is due to Frobenius (1896), at least for $q$
   prime. It can be proven simply by examining a character table, as in
   \cite{MR1691549}.
   The same procedure gives analogues of the same result for other
   groups of Lie type. Alternatively, there is a very nice elementary proof for
   $q$ prime, to be found, for example, in \cite[Lemma 1.3.3]{MR3309986}.
   \begin{prop}\label{prop:froby}
     Let $G = \SL_2(\mathbb{F}_q)$, $q = p^\alpha$.
     Then every non-trivial complex representation of $G$ has dimension
      $\geq (q-1)/2$.
   \end{prop}

   We recall that the adjacency operator $\mathscr{A}$ on a Cayley
   graph $\Gamma(G,A)$ is the linear
   operator that takes a function $f:V\to \mathbb{C}$
   to the function $\mathscr{A} f:V\to \mathbb{C}$ given by
   \begin{equation}\label{eq:adnovdef}\mathscr{A} f(g)
     = \frac{1}{|A|} \sum_{a\in A} f(a g).\end{equation}
     Assume, as usual, that $A = A^{-1}$. Then
     $\mathscr{A}$ is symmetric and all its eigenvalues are real:
   \[ \dotsc \leq \nu_2 \leq \nu_1 \leq \nu_0 = 1.\]
   The largest eigenvalue $\nu_0$ corresponds to the eigenspace of
   constant functions.

   \begin{prob}
     Show that no eigenvalue $\nu$ can be larger than $1$. Hint:
     assume $\nu>1$, and show, using (\ref{eq:adnovdef}), that, for
     $g$ such that $|f(g)|$ is maximal, the equation 
     $\mathscr{A} f(g) = \nu f(g)$ leads to a contradiction.
   \end{prob}
   
   By an {\em eigenspace} of $\mathscr{A}$ we mean, of course, the vector space
     consisting of functions $f$ such that $\mathscr{A} f = \nu f$ for
     some fixed eigenvalue $\nu$. It is clear from the definition
     that every eigenspace of $\mathscr{A}$ is invariant under the action
     of $G$ by multiplication on the right. Hence, an eigenspace of
     $\mathscr{A}$ is a complex
     representation of $G$ -- and it can be trivial only if it is the
     eigenspace of constant functions, i.e., the
     eigenspace corresponding to $\nu_0$. 
     Thus, by Prop.~\ref{prop:froby},
     all other eigenvalues have multiplicity $\geq (q-1)/2$.

   The idea now is to obtain a spectral gap, i.e., a non-trivial upper bound on
   $\nu_j$, $j>0$. It is standard to use the fact that the trace of a power
   $\mathscr{A}^r$ of an adjacency operator $\mathscr{A}$ can be expressed
   in two ways: as a the number of cycles of length $r$ in the graph
   $\Gamma(G,A)$ (multiplied by $1/|A|^r$), and as the sum of the $r$th
   powers of the eigenvalues of $\mathscr{A}$. In our case, for $r=2$, this
   gives us
      \begin{equation}\label{eq:gotra}
     \frac{|G| |A|}{|A|^2} = \sum_j \nu_j^2 \geq \frac{q-1}{2}
   \nu_j^2,\end{equation}
   for any $j\geq 1$,  and hence
   \begin{equation}\label{eq:himult}
     |\nu_j|\leq \sqrt{\frac{|G|/|A|}{(q-1)/2}}.\end{equation}
   This is a very low upper bound when $|A|$ is large. 
   This means that a few applications of the operator $\mathscr{A}$
   are enough to render any function almost uniform, since any component
   orthogonal to the space of constant functions is multiplied by
   some $\nu_j$, $j\geq 1$, at every step. The following proof puts
   in practice this observation efficiently.
   
   \begin{prop}[\cite{MR2800484}]\label{prop:diplo}
     Let $G = \SL_2(\mathbb{F}_q)$, $q=p^\alpha$. Let $A\subset G$,
     $A=A^{-1}$.
     Assume $|A|\geq 2 |G|^{8/9}$. Then
\[A^3 = G.\]
   \end{prop}
   Actually, \cite{MR2800484} proves this result without the assumption
   $A=A^{-1}$.
   We need $A=A^{-1}$ for
     $\mathscr{A}$ to be a symmetric operator,
   but, thanks to \cite{MR2410393}, essentially the same
     argument works in the case $A\ne A^{-1}$.
   \begin{proof}
          Suppose there is a $g\in G$ such that $g\notin A^3$.
     Then the scalar product
     \[\begin{aligned}\langle \mathscr{A} 1_A, 1_{g A}\rangle &=
     \langle \mathscr{A} 1_A, 1_{g A}\rangle =
     \sum_{x\in G} (\mathscr{A} 1_A)(x) \cdot 1_{g A}(x) \\ &=
     \frac{1}{|A|} \sum_{x\in G} \sum_{a\in G} 1_A(a x) \cdot 1_{g A}(x)
     \end{aligned}\]
     equals $0$, as otherwise there is an $x\in g A$ and
     an $a\in A$ such that $a x \in A$, and that would imply 
     $g\in A^{-1} A A^{-1} = A^3$.

     Since $\mathscr{A}$ is symmetric, it has full spectrum, that is,
     there exists a system of $n=|G|$ orthonormal eigenvectors
     $v_0, v_1,\dotsc$
     of $\mathscr{A}$.
     Here $v_0$ is the constant function satisfying
     $\langle v_0,v_0\rangle=1$, that is, the constant function taking
     the value $1/\sqrt{|G|}$ everywhere. Then
\[\begin{aligned}
\langle \mathscr{A} 1_A, 1_{g A}\rangle &=
\langle \sum_{j\geq 0} \nu_j \langle 1_A, v_j\rangle v_j, 1_{g A}\rangle
\\ &= \nu_0 \langle 1_A,v_0\rangle \langle v_0, 1_{g A}\rangle +
\sum_{j>0} \nu_j \langle 1_A, v_j\rangle \langle v_j, 1_{g A}\rangle
.\end{aligned}\] 
%Clearly, since $\langle v_0,v_0\rangle=1$, $v_0$ is the constant
%function that takes the value
%$1/\sqrt{|G|}$ everywhere.
Now
\[\nu_0 \langle 1_A,v_0\rangle \langle v_0, 1_{g A}\rangle =
1\cdot \frac{|A|}{\sqrt{|G|}}\cdot \frac{|g A|}{\sqrt{|G|}} = 
\frac{|A|^2}{|G|}.\]
At the same time, by (\ref{eq:himult}) and
Cauchy-Schwarz,
\[\begin{aligned} \left|\sum_{j>0} \nu_j \langle 1_A, v_j\rangle \langle v_j, 1_{g A}\rangle\right|
&\leq
\sqrt{\frac{2 |G|/|A|}{q-1}}
\sqrt{\sum_{j\geq 1}  |\langle 1_A, v_j\rangle|^2}
\sqrt{\sum_{j\geq 1} |\langle v_j, 1_{g A}\rangle|^2}\\
&\leq 
\sqrt{\frac{2 |G|/|A|}{q-1}} |1_A|_2 |1_{g A}|_2
=  \sqrt{\frac{2|G| |A|}{q-1}} .
\end{aligned}\]
Since $|G|= q (q^2-1)$, we see that $|A|\geq 2 |G|^{8/9}$ implies
\[\frac{|A|^2}{|G|} > \sqrt{\frac{2 |G| |A|}{q-1}},\] and thus
$\langle \mathscr{A} 1_A, 1_{g A}\rangle > 0$. Contradiction.
   \end{proof}
   
 \section{Further perspectives and open problems}\label{sec:further}

\subsection{Expansion, random walks and the affine sieve}\label{sec:exprawsie}

Let $G$ be a group, $A\subset G$, $A = A^{-1}$.
As we saw in \S \ref{subs:whatisg}, the adjacency operator
$\mathscr{A}$ has full real spectrum, and we can define what it means for
the graph $\Gamma(G,A)$ to be a {\em $\delta$-spectral expander}, or
simply an {\em $\delta$-expander}.
An infinite family of graphs $\Gamma(G_i,A_i)$
is called an {\em expander family} if there is an $\epsilon>0$ such that
every $\Gamma(G_i,A_i)$ is an $\epsilon$-expander. Of particular interest are
expander families with $|A_i|$ bounded.

Using
Thm.~\ref{thm:main08},
Bourgain and Gamburd proved the following result
\cite{MR2415383}.
\begin{theorem}\label{thm:bg}
  Let $A_0\subset \SL(\mathbb{Z})$. Assume that $A_0$ is not
  contained in any proper algebraic subgroup of $\SL_2$. Then
  \begin{equation}\label{eq:hustu}
    \{\Gamma(\SL_2(\mathbb{Z}/p\mathbb{Z}), A_0 \mo p)\}_{p>C,
    \text{$p$ prime}}\end{equation}
  is an expander family for some constant $C$.
\end{theorem}
The proof also involves Proposition \ref{prop:froby}
(applied as in \cite{SarnakXue})
as well as a non-commutative version \cite{MR2501249}
of the Balog-Gowers-Szemer\'edi theorem
from additive combinatorics.
There are by now wide-ranging
generalizations of Thm.~\ref{thm:bg}; see, e.g., \cite{SGV}.

%Here the condition that $A_0$ is not contained in any proper algebraic
%subgroup of $\SL_2$ guarantees that $A_0 \mo p$ generates
%$\SL_2(\mathbb{Z}/p\mathbb{Z})$ for $p$ greater than a constant $C$;
%this is proven in much greater generality in \cite{MR763908}, \cite{MR735226},
%\cite{MR880952}, \cite{MR1329903}.

A {\em random walk} on a graph is what it sounds like: we start at a vertex
$v_0$, and at every step we move to one of the $d$ neighbors of the vertex
we are at -- choosing any one of them with probability $1/d$. For convenience
we work with a {\em lazy random walk}: at every step, we decide to stay
where we are with probability $1/2$, and to move to a neighbor with probability
$1/2d$. The {\em mixing time} is the number of steps it takes for ending point
of a lazy random walk to become almost equidistributed (where ``almost''
is understood in any reasonable metric). In an $\epsilon$-expander graph
$\Gamma(G,A)$, the mixing time is $O_\epsilon(\log |G|)$, i.e., about as small
as it could be: it is easy to see that, for $|A|$ bounded, the mixing
time (and even the
diameter) has to be $\gg \log |G|$.

\begin{prob}
  Let $G$ be a group, $A\subset G$, $A=A^{-1}$, $\langle A\rangle = G$.
  Let $\mathscr{A}$ be the adjacency operator on the Cayley graph. 
  \begin{enumerate}
  \item Take a lazy random walk with $k$ steps on the Cayley graph,
    starting at the identity $e$.
    Show that the probability of your final position is given by the function
    $\phi_k = ((\mathscr{A}+I)/2)^k \delta_e$, where $\delta_e:G\to \mathbb{C}$
    is the function taking the value $1$ at $e$ and $0$ elsewhere.
  \item Write $\delta_e$ as a linear combination
    $\delta_e = \sum_j c_j v_j$, where each $v_j$ is an eigenvector of
    $\mathscr{A}$. What is the coefficient in front of the constant
    eigenvector $v_0$? What is $((\mathscr{A}+I)/2)^k \delta_e$,
    as a linear combination of the eigenvectors $v_j$?
  \item Assume $\Gamma(G,A)$ is a $\delta$-expander. Show that, for
    $k\geq (2 C/\delta) \log |G|$, $C\geq 1$,
    the probability distribution $\phi_k$ is
    nearly uniform in both the $\ell^2$- and the $\ell^\infty$-norms:
    \[\sum_{g\in G} \left| \phi_k(g) - \frac{1}{|G|}\right|^2 \leq \frac{1}{|G|^{C}},\]
    \[\max_{g\in G} \left| \phi_k(g) - \frac{1}{|G|}\right| \leq \frac{1}{|G|^{C-1}}.\]
    That is to say, the mixing time with respect to either
    the $\ell^2$- or the $\ell^\infty$-norms is $\ll (1/\delta) \log |G|$.
    \end{enumerate}
\end{prob}

Thus, Thm.~\ref{thm:bg} gives us small mixing times. This fact has made
the {\em affine sieve} possible \cite{MR2587341}. The affine sieve is an
analogue
of classical sieve methods; they are recast as sieves based on the natural action
of $\mathbb{Z}$ on $\mathbb{Z}$, whereas a general affine sieve considers 
the actions of other groups, such as $\SL_2(\mathbb{Z})$.

Expansion had been shown before for some specific $A_0$. In particular,
when $A_0$ generates $\SL_2(\mathbb{Z})$ (or a subgroup of finite index before)
then the fact that (\ref{eq:hustu}) is an expander graph can be derived
from the {\em Selberg spectral gap} \cite{MR0182610}, i.e., the fact
that the Laplacian on the quotient $\SL_2(\mathbb{Z})\backslash \mathbb{H}$
of the upper half plane $\mathbb{H}$ has a spectral gap. Nowadays, one
can go in the opposite direction: spectral gaps on more general quotients can
be proven using Thm.~\ref{thm:bg} \cite{MR2892611}.

Let us finish this discussion by saying that it is generally held to be
plausible that the family of {\em all} Cayley graphs of $\SL_2(\mathbb{Z}/p\mathbb{Z})$, for all $p$, is an expander family; in other words, there may
be an $\epsilon>0$ such that, for every prime $p$ and every
generator $A$ of $\SL_2(\mathbb{Z}/p\mathbb{Z})$, the graph
$\Gamma(\SL_2(\mathbb{Z}/p\mathbb{Z}),A)$ is an $\epsilon$-expander.
This statement has seemed plausible at least since \cite{MR1203870},
but proving it is an open problem believed to be very hard. It has
been shown that there exists a thin family of primes such that the
statement is true if those primes are omitted \cite{MR2746951}.

\subsection{Algorithmic and probabilistic questions}

It is one thing to show that the diameter of a group $G$ is small,
that is, to show that every element of $G$ can be written as short
word on any set of generators $A$. (By a {\em word} on $A$ we mean
a product of elements 
of $A \cup A^{-1}$.) It is quite another to be able to find that word
-- reasonably quickly, it is understood.

Larsen \cite{MR1976231} gave a probabilistic algorithm
that expresses an arbitrary $g\in \SL_2(\mathbb{Z}/p\mathbb{Z})$
as a word of length $O(\log p \log \log p)$ in the generators
\[A = \left\{\left(\begin{matrix} 1 & 1\\ 0 & 1\end{matrix}\right),
  \left(\begin{matrix} 1 & 0\\ 1 & 1\end{matrix}\right)
  \right\}\]
in time $(\log p)^{O(1)}$. No algorithm is known for arbitrary generators
of $\SL_2(\mathbb{Z}/p\mathbb{Z})$. Neither do we have an algorithm for
finding short words on
arbitrary generators of finite simple groups in any other family.

%We say that $G$ is {\em solvable} if it has a subnormal series with
%$H_{i+1}/H_i$ abelian for every $i$. Being nilpotent is a stronger
%condition: a group $G$ is nilpotent if it has a normal series such that,
%for every $i$, $H_i$ is normal in $G$ and $H_{i+1}/H_i$ lies in the
%center of $G/H_i$.

%A nilpotent group can often be thought of as being
%``almost abelian''; the present context is no exception.

Another question is what happens when $g_1, g_2$ are random elements of a
group $G$. For several kinds of groups (linear algebraic, $\Alt(n)$)
it is known that, with probability tending to one, $g$ and $h$ generate
$G$. What is the diameter
of the Cayley graph of $G$ with respect to $\{g,h\}$ likely to be?
For $G=\SL_2(\mathbb{F}_p)$, it is known that it is $O(\log |G|)$
with probability tending to one (by \cite{MR2532876} taken together with
Thm.~\ref{thm:main08}). For $\Alt(n)$, it is known to be
$O(n^2 (\log n)^{O(1)})$ with probability tending to one \cite{MR3272386}.
Is it actually $O(n (\log n)^{O(1)})$, or even $O(n \log n)$, with probability
tending to one?

One can combine algorithmic and probabilistic questions. The proof in 
\cite{BBS04} (supplemented by \cite{BH}) yields a probabilistic algorithm that,
for a proportion $\to 1$ (as $n\to \infty$) of all pairs of elements
$g_1$, $g_2$ of $\Alt(n)$, expresses any given element $g$ of $\Alt(n)$
as a word of polynomial length on $g_1$ and $g_2$, and does so in (Las Vegas) polynomial time.
(If the algorithm will fail for a given pair $(g_1,g_2)$, it states so
at an initial stage taking polynomial time.) The procedure in
\cite{MR3272386} gives a probabilistic algorithm that finds a word
of length $O(n^2 (\log n)^{O(1)})$ in time $O(n^2 (\log n)^{O(1)})$ for
a proportion $\to 1$ of all pairs $g_1$, $g_2$ and $g$ arbitrary,
as is sketched in
\cite[App.~B]{MR3272386}.

No analogous algorithm is known over
$\SL_2(\mathbb{F}_q)$, or for any other simple group of Lie type; we
do not know how to express an arbitrary element of $\SL_2(\mathbb{F}_q)$
as a word of length $(\log q)^{O(1)}$ on a random pair of generators of $G$
in time $(\log q)^{O(1)}$. 

%We will focus on the case of $G$ non-abelian, and, in particular, on
%the case of $G$ a group of Lie type, such as $\SL_2$ over an arbitrary field
%$K$. The case of $K$ finite can be particularly hard, in that we cannot
%be helped by the topology of $\mathbb{R}$ or $\mathbb{C}$, say.

\subsection{Final remarks}
Let us briefly mention some links with other areas.

{\em Group classification.} It is by now clear that it is useful to look at a particular kind of result
in group classification: the kind that was developed so as to avoid casework,
and to do without the Classification of Finite Simple Groups. (The Classification is now generally accepted, but this was not
always the case, and it is still sometimes felt to be better to prove something
without it than with it; what we are about to see gives itself some
validation to this viewpoint.) While results proven without the Classification
are sometimes weaker than others, they are also more robust. Classifying
subgroups of a finite group $G$ is the same as classifying subsets $A\subset G$
such that $e\in A$ and $|A A| = |A|$. Some Classification-free classification methods can be
adapted to help in classifying subsets $A\subset G$ such that $e\in A$ and
$|A A A|\leq |A|^{1+\delta}$ -- in other words, precisely what we are studying.
It is in this way that \cite{LP} was useful in \cite{BGT},
and \cite{Bab82}, \cite{Pyb93} were useful in \cite{MR3152942}.

{\em Model theory.} Model theory
is essentially a branch of logic with applications
to algebraic structures.
Hrushovski and his collaborators \cite{MR1329903}, \cite{MR2436141},
\cite{MR2833482} have used model theory
to study subgroups of algebraic groups. This
was influenced by Larsen-Pink \cite{LP}, and also served to
explain
it. In turn, \cite{MR2833482} influenced later work, especially
\cite{BGTgen}.

{\em Permutation-group algorithms.} Much work on permutation groups has been
algorithmic in nature. Here a standard reference is \cite{MR1970241}.
A good example is a problem we mentioned before -- that
of bounding the diameter of $\Sym(n)$ with
respect to a random pair of generators; the approach in \cite{BBS04}
combines probabilistic and algorithmic ideas -- as does \cite{MR3272386}, which
builds on \cite{BBS04}, and as, for that matter, does \cite{MR3152942}. The
reference \cite{MR2466937} treats several of the relevant probabilistic tools.

{\em Geometric group theory.} Here much work remains to be done. Geometric
group theory, while still a relatively new field, is considerably older than
the approach followed in these notes. It is clear that there is a connection,
but it has not yet been fully explored. Here it is particularly worth
remarking that \cite{MR2833482} gave a new proof of Gromov's theorem by means
of the study of sets $A$ that grow slowly in the sense used in these notes.

%geometric group theory
 \bibliographystyle{alpha}
\bibliography{aws}
\end{document}